\documentclass{amsart}[10pt,a4paper,twoside]
\usepackage[T1]{fontenc}
\usepackage[utf8]{inputenc}
\usepackage[english]{babel}
\usepackage{amsmath,amssymb,fancyhdr,enumerate,bbm}
\usepackage[all]{xy}
\usepackage[active]{srcltx}
\usepackage{xcolor}
\usepackage[hypertexnames=false]{hyperref}

\title[Degrees over Perfect Fields]{Degrees of Irreducible Morphisms
  over Perfect Fields} \author[C. Chaio]{Claudia Chaio}
\address[Claudia Chaio]{Centro Marplatense de Investigaciones
  Matemáticas, FCEyN, Universidad Nacional de Mar del Plata,
  CONICET. Funes 3350, 7600 Mar del Plata, Argentina}
\email{algonzal@mdp.edu.ar}

\author[P. Le Meur]{Patrick Le Meur}
\address[Patrick Le Meur]{
Laboratoire de Math\'ematiques, Universit\'e Blaise Pascal \&
  CNRS, Complexe Scientifique Les Cézeaux, BP 80026, 63171 Aubi\`ere
  cedex, France}
\curraddr{Universit\'e Paris Diderot, Sorbonne Universit\'e, CNRS, Institut de
   Math\'ematiques de Jussieu-Paris Rive Gauche, IMJ-PRG, F-75013, Paris, France}
\email{patrick.le-meur@imj-prg.fr}

\author[S. Trepode]{Sonia Trepode}
\address[Sonia Trepode]{Centro
  Marplatense de Investigaciones Matemáticas, FCEyN, Universidad
  Nacional de Mar del Plata,CONICET. Funes 3350, 7600 Mar del Plata,
  Argentina}
\email{strepode@mdp.edu.ar}
\date{\today}

\newtheorem{prop}{Proposition}[subsection]
\newtheorem{Thm}{Theorem}
\newtheorem{cor}[prop]{Corollary}
\newtheorem{lem}[prop]{Lemma}
\newtheorem{thm}[prop]{Theorem}

\newtheorem{ex}[prop]{Example}
\newtheorem{rem}[prop]{Remark}

\def\k{\mathbbm k}

\subjclass[2010]{16G10; 16G60; 16G70}


\begin{document}

\begin{abstract}
  The module category of any artin algebra is filtered by the powers
  of its radical, thus defining an associated graded category. As an
  extension of the degree of irreducible morphisms, this text
  introduces the degree of morphisms in the module category.  When the
  ground ring is a perfect field, and the given morphism behaves
  nicely with respect to covering theory (as do irreducible morphisms
  with indecomposable domain or indecomposable codomain), it is shown
  that the degree of the morphism is finite if and only if its induced
  functor has a representable kernel. This gives a generalisation of
  Igusa and Todorov result, about irreducible morphisms with finite
  left degree and over an algebraically closed field.  As a corollary,
  generalisations of known results on the degrees of irreducible
  morphisms over perfect fields are given. Finally, this study is
  applied to the composition of paths of irreducible morphisms in
  relationship to the powers of the radical.
\end{abstract}

\maketitle


\section{Introduction}

Let $A$ be an artin algebra over an artin commutative ring $\k$ and $\mathrm{mod}\,A$ its category of finitely generated modules. The radical
($\mathrm{rad}$) of $\mathrm{mod}\,A$ is related to many useful tools to
understand this category. In particular, it is deeply connected to
Auslander-Reiten theory based
on irreducible morphisms. Recall that $\mathrm{rad}$ denotes the ideal in
$\mathrm{mod}\,A$ generated by non-isomorphisms between indecomposable modules. Its powers $\mathrm{rad}^\ell$ ($\ell\geqslant 0$) are defined inductively by $\mathrm{rad}^0=\mathrm{mod}\,A$ and $\mathrm{rad}^{\ell+1}=\mathrm{rad}^\ell\cdot \mathrm{rad}=\mathrm{rad}\cdot \mathrm{rad}^\ell$. The infinite radical $\mathrm{rad}^\infty$ is defined as $\mathrm{rad}^{\infty}=\cap_{\ell\geqslant 0}\mathrm{rad}^\ell$. In particular, a
morphism in $\mathrm{mod}\,A$ between
indecomposable modules is irreducible if and only if it lies in $\mathrm{rad}\backslash \mathrm{rad}^2$ or, equivalently, it lies in $\mathrm{rad}$ and its image in $\mathrm{rad}/\mathrm{rad}^2$ is non-zero.

The purpose of this article is to generalise the results on degrees and 
composition of irreducible morphisms in $\mathrm{mod}\,A$ proven in \cite{CLT1} to the context 
of finite-dimensional $\k$-algebras over a perfect field.

These new results are presented in connection with the filtration
$(\mathrm{rad}^\ell)_{\ell\geqslant 0}$ of $\mathrm{mod}\,A$, with a point
of view that may be interesting for future investigations.  Results on
compositions have already provided useful information on the
Auslander-Reiten structure of $A$ (see for instance \cite{L2}). Such
results involve the degree of irreducible morphisms, introduced by
S. Liu in \cite{L}. These are defined in terms of indices $n$ such
that the composition of the given irreducible morphism with some
morphism in $\mathrm{rad}^n$ and not in $\mathrm{rad}^{n+1}$ lies in
$\mathrm{rad}^{n+2}$.

K.~Igusa and G.~Todorov proved in \cite[Theorem 6.2]{MR748231} that, when
$A$ is an artin algebra of finite representation type, or a finite-dimensional
algebra over an algebraically closed field, and when $f\colon X\to Y$
is an irreducible morphism with $X$ or else $Y$
indecomposable such that the left degree of $f$ is finite then the
kernel of the functor $\mathrm{Hom}_A(-, f)$ is representable.

One of the aims of this paper is to extend such results when $A$ is a
finite-dimensional algebra over a perfect field. More precisely, the
previously mentioned result of K.~Igusa and G.~Todorov asserts that,
under the same conditions, if the inclusion morphism
$i\colon \mathrm{Ker}(f) \to X$ lies in
$\mathrm{rad}^n\backslash \mathrm{rad}^{n+1}$ for some integer $n$, then the
following sequence is exact for all integers $\ell\geqslant 1$, and
indecomposables $Z\in \mathrm{mod}\,A$,
\[
0 \to
\frac{\mathrm{rad}^{\ell-n}}{\mathrm{rad}^{\ell-n+1}}(Z,\mathrm{Ker}(f))
\xrightarrow{-\cdot i }
\frac{\mathrm{rad}^\ell}{\mathrm{rad}^{\ell+1}}(Z,X)
\xrightarrow{-\cdot f}
\frac{\mathrm{rad}^{\ell+1}}{\mathrm{rad}^{\ell+2}}(Z,Y)\,.
\]
In order to generalise the above result, we introduce, following
S. Liu, the notion of left and right degree of any morphism $f$ from
$X$ to $Y$ in $\mathrm{mod}\,A$ such that
$f \in \mathrm{rad}^{n} \backslash \mathrm{rad}^{n+1}$, for some positive
integer $n$.

More precisely, we define $d_\ell(f)$ to be the least integer $m$ such
that there exists $(Z,g)$ with $Z\in \mathrm{mod}\,A$ indecomposable
and $g\colon Z\to X$ lying in
$\mathrm{rad}^m\backslash \mathrm{rad}^{m+1}$ and such that
$fg\in \mathrm{rad}^{m+n+1}$ (and $\infty$ if such $(Z,g)$ does not
exist). In other words, $d_\ell(f)$ is the least integer $d$ such that
the morphism of functors
$\frac{\mathrm{rad}^d(-,X)}{\mathrm{rad}^{d+1}(-,X)} \to \frac{\mathrm{rad}^{n+d}(-,Y)}{\mathrm{rad}^{n+d+1}(-,Y)}$
induced by $f$ is not a monomorphism.  Note that, when $f$ is replaced
by any $f'$ such that $f-f'\in \mathrm{rad}^{n+1}$, then the following
morphism
\begin{equation}
  \label{eq:2}
  \oplus_{d\geqslant 0} \frac{\mathrm{rad}^d(-,X)}{\mathrm{rad}^{d+1}(-,X)} \longrightarrow \oplus_{d\geqslant 0} \frac{\mathrm{rad}^{n+d}(-,Y)} {\mathrm{rad}^{n+d+1}(-,Y)}
\end{equation}
between contravariant functors from $\mathrm{mod}\,A$ to the category of graded
$\k$-vector spaces remains unchanged, and hence $d_\ell(f)=d_\ell(f')$.
The \emph{right degree} $d_r(f)$ is defined dually.

The results of this article are based on functors with the covering
property, which exist for any algebra $A$ over a perfect field $\k$,
see (\cite{CLT2}). Hence,
\begin{center}
from now on, $\k$ is assumed to be a perfect field.
\end{center}

Given an Auslander-Reiten component $\Gamma$ of $\mathrm{mod}\,A$, given
a $\k$-linear category with length $\mathcal C$ and given a
well-behaved functor $F \colon \mathcal C \to \k(\Gamma)$, a morphism
$X\to Y$ in $\mathrm{mod}\,A$ is called \emph{homogeneous} if there exist
tuples $(x_s)_s$ and $(y_t)_t$ of objects in $\mathcal C$ and
morphisms $(\varphi_{s,t})_{s,t}\in \oplus_{s,t}\mathcal C(x_s,y_t)$
such that $X=\oplus_x Fx_s$, $Y=\oplus_t Fy_t$ and for every $s,t$,
the component $Fx_s\to Fy_t$ of the given morphism $X\to Y$ is equal
to $F(\varphi_{s,t})$. In such a case, the morphism is called
homogeneous with respect to $(y_t)_t$ (or to $(x_s)_s$) if it is
necessary to work with a predetermined decomposition $Y=\oplus_tFy_t$
(or $X=\oplus_sFx_s$, respectively). Note that, if the given morphism
$X\to Y$ lies in $\mathrm{rad}^m$ for some $m$, then it is necessary that
$\varphi_{s,t}\in \mathcal R^m\mathcal C$ for every $s,t$.

Given $d\in \mathbb N$, a morphism $f\colon X\to Y$ is called homogeneous \emph{up to $\mathrm{rad}^{d+1}$} if it is the sum of two morphisms from $X$ to $Y$,
the former being homogeneous and the latter lying in $\mathrm{rad}^{d+1}$.

Now, we state the first result of the article.

\begin{Thm}[Proposition~\ref{sec:kern-char-homog-1}]
  \label{sec:introduction-1}
Let $f\colon X\to Y$ be a morphism in $\mathrm{mod}\,A$. Assume that
\begin{itemize}
\item $f$ lies in $\mathrm{rad}^d\backslash \mathrm{rad}^{d+1}$ and is homogeneous up to $\mathrm{rad}^d$,
\item $f$ has finite left degree denoted by $n$.
\end{itemize}
Then,
\begin{enumerate}
\item there exists $f'\colon X\to Y$ such that
  $f-f'\in \mathrm{rad}^{d+1}$ and such that the inclusion morphism
  $i\colon \mathrm{Ker}(f')\to X$ lies in
  $\mathrm{rad}^n\backslash\mathrm{rad}^{n+1}$,
\item there exists a direct sum decomposition
  $\mathrm{Ker}(f') = (\oplus_{m\geqslant n}K^{(m)})\oplus K^{(\infty)}$
  such that the restriction of the inclusion morphism $i$ to any
  indecomposable direct summand of $K^{(m)}$ (or, to $K^{(\infty)}$)
  lies in $\mathrm{rad}^m\backslash\mathrm{rad}^{m+1}$ (or, to
  $\mathrm{rad}^\infty$, respectively),
  \item for every integer $\ell\geqslant n$ and every indecomposable
    $Z\in \mathrm{mod}\,A$, the following sequence
        \[
    0\to \bigoplus\limits_{n\leqslant m \leqslant \ell} \frac{\mathrm{rad}^{\ell-m}}{\mathrm{rad}^{\ell-m+1}}(Z,K^{(m)})
      \xrightarrow{-\cdot i}
      \frac{\mathrm{rad}^{\ell}}{\mathrm{rad}^{\ell+1}}(Z,X)
      \xrightarrow{-\cdot f}
      \frac{\mathrm{rad}^{\ell+d}}{\mathrm{rad}^{\ell+d+1}}(Z,Y)\,
      \]
      \noindent is exact.
  \end{enumerate}
\end{Thm}

A morphism in $\mathrm{rad}^d\backslash\mathrm{rad}^{d+1}$ is not necessarily
homogeneous up to $\mathrm{rad}^{d+1}$, but is a sum of such
morphisms.
Any irreducible morphism with indecomposable domain (or
codomain) is homogeneous up to $\mathrm{rad}^2$.

As a consequence of Theorem A, it is possible to extend the above
mentioned result of K.~Igusa and G.~Todorov to the context of
finite-dimensional algebras over perfect fields.

\begin{Thm}[Theorem~\ref{sec:finite-left-degree-1}]
  \label{sec:introduction-2}
  Let $f\colon X\to Y$ be an irreducible morphism. Assume that $X$, or
  else $Y$, is indecomposable and that $f$ has finite left degree
  denoted by $n$. Then, there exists $f'\colon X\to Y$ irreducible
  such that
  \begin{enumerate}
  \item $f-f'\in \mathrm{rad}^2$,
  \item the inclusion morphism $i\colon \mathrm{Ker}(f')\to X$ lies in
    $\mathrm{rad}^n\backslash\mathrm{rad}^{n+1}$,
  \item for every integer $\ell\geqslant n$ and every
    indecomposable $Z\in \mathrm{mod}\,A$, the following sequence is exact
    \[
    0 \to
    \frac{\mathrm{rad}^{\ell-n}}{\mathrm{rad}^{\ell-n+1}}(Z,\mathrm{Ker}(f'))
    \xrightarrow{-\cdot i }
    \frac{\mathrm{rad}^\ell}{\mathrm{rad}^{\ell+1}}(Z,X)
    \xrightarrow{-\cdot f}
    \frac{\mathrm{rad}^{\ell+1}}{\mathrm{rad}^{\ell+2}}(Z,Y)\,
    \]
    is exact. In addition, if $f$ is freely irreducible, then $f'$ may
    be chosen equal to $f$.
  \end{enumerate}
\end{Thm}

Following \cite[Section 2.5]{CLT2}, a morphism
$f\colon X \to \oplus_{i=1}^r X_i^{n_i}$ with $X,X_i$ indecomposable
is called \emph{freely irreducible} if, for every
$i\in\{1,\ldots,r\}$, the $n_i$-tuple
$(\overline{f_{i,1}},\ldots,\overline{f_{i,n_i}})$ of residue classes
in $\mathrm{irr}(X,X_i)$ is free over
$\kappa_X\otimes_\k\kappa_{X_i}^{\mathrm{op}}$. Note that freely
irreducible morphisms are particular cases of strongly irreducible
morphisms introduced in \cite{CLT2}. In particular, when $X$ is
indecomposable, $f$ is freely irreducible if the division algebra
$\mathrm{End}_A(X)/\mathrm{rad}(X,X)$ is trivial. When $\k$ is an
algebraically closed field, the irreducible morphism in
Theorem~\ref{sec:introduction-2} is automatically freely
irreducible. In this setting, Statements (1) and (2) of
Theorem~\ref{sec:introduction-2} where proved in \cite{CLT1} as well
as the particular case $\ell=n$ of Statement (3). Actually, it is
possible to derive from the above theorem extensions of the results in
\cite{CLT1} (there, $\k$ is an algebraically closed field) to finite
dimensional algebras over perfect fields. For instance
\begin{itemize}
\item if $f\colon X\to Y$ is a freely irreducible morphism where
  $X$, or else $Y$ is indecomposable, then $d_\ell(f)$ is finite (equal
  to some given $n$) if and only if $f$ is not a monomorphism and the
  inclusion morphism $\mathrm{Ker}(f)\to X$ lies in
  $\mathrm{rad}^n\backslash\mathrm{rad}^{n+1}$ (see
  Corollary~\ref{sec:spec-freely-irred-1}),
\item given two irreducible morphisms $f_1,f_2\colon X\to Y$ where $X$
  is indecomposable and the division algebra $\mathrm{End}_A(X)/\mathrm{rad}(X,X)$ is trivial, if
  $d_\ell(f_1)<\infty$, then $d_\ell(f_1)=d_\ell(f_2)$ and $\mathrm{Ker}(f_1)\simeq \mathrm{Ker}(f_2)$,
\item the algebra $A$ is of finite representation type if and only if,
  for every indecomposable injective $I\in \mathrm{mod}\,A$, the
  irreducible morphism $I\to I/\mathrm{soc}(I)$ has finite left degree
  (see Theorem~\ref{sec:finite-left-degree-1} for a richer statement).
\end{itemize}

Finally, one of the original motivations for studying the kernel
of~(\ref{eq:2}) was to determine when the composition of $n$
irreducible morphisms between indecomposable $A$-modules lies in $\mathrm{rad}^{n+1}$. A first approach of this problem is given in
\cite{CLT1} when $\k$ is algebraically closed and extended to the
case where $\k$ is a perfect field in \cite{CLT2}. The above results
provide another approach to this problem.
\begin{Thm}
  \label{sec:introduction}
  Let $A$ be a finite dimensional $\k$-algebra over a perfect
  field. Let
  $X_0 \xrightarrow{f_1} X_1 \to \cdots \to X_{n-1}\xrightarrow{f_n}
  X_n$
  be a chain of irreducible morphisms between indecomposable
  $A$-modules. For each $t$, let $f_t'\colon X_{t-1}\to X_t$ be such
  as $f'$ in Theorem B when $f=f_t$. Consider the following
  assertions.
  \begin{enumerate}[(i)]
  \item  $f_1 \cdots f_n \in \mathrm{rad}^{n+1}$.
  \item There exists $t\in \{1,\ldots,n\}$ such that
    $d_\ell(f_t)\leqslant t-1$, and there exists $h\in \mathrm{rad}^{t-1-d_\ell(f_t)}(X_0,\mathrm{Ker}(f'_t))$ not lying in
    $\mathrm{rad}^{t-d_\ell(f_t)}$, and such that $f_1\cdots f_{t-1}-h i \in \mathrm{rad}^t$ (where $i\colon \mathrm{Ker}(f'_t)\to X_{t-1}$ is the
    inclusion morphism).
  \item There exists $t\in \{1,\ldots,n\}$ such that
    $d_l(f_t) \leq t-1$, there exists a path of length
    $t-1-d_\ell(f_t)$ from $X_0$ to $\mathrm{Ker}(f'_t)$ and with nonzero
    composition, and there exists a path
    $X_0 \rightarrow X_1 \rightarrow \dots \rightarrow
    X_{t-1}\rightarrow X_{t}$ with zero composition.
  \end{enumerate}
  Then $(i)$ and $(ii)$ are equivalent and imply $(iii)$.
  If, moreover, $\mathrm{dim}_{\k}\mathrm{irr}(X_{s-1},X_s)=1$ for every $s\in \{1,\ldots,t\}$, then $(iii)$
  implies $(i)$ and $(ii)$.
\end{Thm}

The text is organised as follows. Section~\ref{sec:prel} sets some
usual conventions on modules and gives some technical lemmas. It also
introduces categories with length and functors with a covering
property defined on these categories. These notions are essential in
the proofs of the above mentioned results.
Section~\ref{sec:degr-homog-morph} is devoted to proving
Theorem~\ref{sec:introduction-1}. This is applied in
section~\ref{sec:applications1} to irreducible morphisms. There,
Theorem~\ref{sec:introduction-2} and the above mentioned results
related to degrees of irreducible morphisms are proved.  Finally,
Theorem~\ref{sec:introduction} is proved in
Section~\ref{sec:applications1}.

Unless otherwise stated, $\k$ denotes a perfect field and $A$ is a finite
dimensional algebra over $\k$.

\section{Preliminaries}
\label{sec:prel}

This section introduces basic material used in the proofs of the main
results of the text. The conventions on modules are set in
\ref{subsec_modules}. Next, \ref{sec:prop-kern-irred}
 collects a
couple of useful results on kernels of irreducible morphisms. The
proofs of the main results are based on functors with the covering
property defined on categories with length. The latter are defined in
\ref{sec:categ-with-length} and the former are introduced in
\ref{sec:funct-with-cover}.

\subsection{Conventions on modules}
\label{subsec_modules}

Let $\mathrm{ind}\, A$ be a full subcategory of $\mathrm{mod}\, A$ which
contains exactly one representative of each isomorphism class of
indecomposable modules. Given modules $X,Y\in \mathrm{ind}\,A$, the
quotient vector space $\mathrm{rad}(X,Y)/\mathrm{rad}^2(X,Y)$ is denoted by
$\mathrm{irr}(X,Y)$ and called the \emph{space of irreducible morphisms}
from $X$ to $Y$. It is naturally an
$\mathrm{End}_A(X)/\mathrm{rad}(X,X)-\mathrm{End}_A(Y)/\mathrm{rad}(Y,Y)$-bimodule.
The division $\k$-algebra $\mathrm{End}_A(X)/\mathrm{rad}(X,X)$ is denoted
by $\kappa_X$.  Recall that the \emph{Auslander-Reiten quiver of $A$}
is the quiver $\Gamma(\mathrm{mod}\, A)$ with vertices the modules in
$\mathrm{ind}\, A$, such that there is an arrow (and exactly one)
$X\to Y$ if and only if $\mathrm{irr}(X,Y)\neq 0$ for every pair of
vertices $X,Y\in\Gamma$.  The Auslander-Reiten translation is denoted
by $\tau_A=D\mathrm{Tr}$. If $\Gamma$ is a connected component of
$\Gamma(\mathrm{mod}\, A)$ (or an \emph{Auslander-Reiten component}, for
short), the full subcategory of $\mathrm{ind}\, A$ with objects the
modules in $\Gamma$ is denoted by $\mathrm{ind}\,\Gamma$.  Let
$f\colon X\to \oplus_{i=1}^rX_i^{n_i}$ be an irreducible morphism
where $X\in \mathrm{ind}\,A$, $X_1,\ldots,X_r\in \mathrm{ind}\,A$ are
pairwise non isomorphic and $n_1,\ldots,n_r\geqslant 1$. Recall from
the introduction that, $f$ is freely irreducible if, for every
$i\in\{1,\ldots,r\}$, the $n_i$-tuple
$(\overline{f_{i,1}},\ldots,\overline{f_{i,n_i}})$ of residue classes
in $\mathrm{irr}(X,X_i)$ is free over
$\kappa_X\otimes_{\k}\kappa_{X_i}^{\mathrm{op}}$. It is always free
over $\kappa_{X_i}^{\mathrm{op}}$. Note that $f$ is freely irreducible under
any of the following conditions: if $\kappa_X\simeq \k$; or if $\k$ is
an algebraically closed field. With dual considerations is defined
freely irreducible morphisms with indecomposable codomain.

Let $X\to Y$ be an arrow in $\Gamma(\mathrm{mod}\,A)$ with valuation
denoted by $(a,b)$. In particular, $a$ equals
$\mathrm{dim}_{\kappa_X}\mathrm{irr}(X,Y)$. Since
$\kappa_X\otimes_{\k} \kappa_Y^{\mathrm{op}}$ is semisimple
$\k$-algebra (because $\k$ is perfect), $a$ is at least
$\ell(\mathrm{irr}(X,Y))$ and at most
$\mathrm{dim}_{\mathbbm k}\kappa_Y \cdot \ell(\mathrm{irr}(X,Y))$,
where $\ell$ denotes the length over
$\kappa_X \otimes_{\mathbbm k} \kappa_Y^{\mathrm{op}}$. Dual
identities apply to $b$. Note that, if $a=1$ or $b=1$, then
$\ell(\mathrm{irr}(X,Y))=1$.  This occurs when, for instance, the arrow
$X\to Y$ has finite left degree (see \cite[Section 1.6]{L2}).

\medskip

The following lemma will be useful later on.
\begin{lem}
  \label{sec:making-section-from-2}
  Let $\{M_i\}_i$ be a family of indecomposable modules and let
  $M=\oplus_i M_i$. For every $i$, let
  $\{\lambda_{i,j}\}_j$ be a family of morphisms $M_i\to M$. Assume
  that $[\sum_j\lambda_{i,j}\ ;\ i]^T\colon \oplus_i M_i \to M$ is an
  isomorphism. Then, there exists an index $j_i$, for every $i$, such
  that the resulting morphism $[\lambda_{i,j_i}\ ;\ i]^T\colon
  \oplus_i M_i\to M$ is an isomorphism.
\end{lem}
\begin{proof}
  Clearly it suffices to treat the case where $M=M_1\oplus M_2$, where
  $\{\lambda_{1,j}\}_j$ consists of two terms, say $u,v$, and where
  $\{\lambda_{2,j}\}_j$ consists of only one term, say $w$. The
  hypothesis then says that $[u+v,w]^T\colon M_1\oplus M_2\to M$ is an
  isomorphism and the conclusion says that at least one of the two
  morphisms $[u,w]^T$ or $[v,w]^T$ is an isomorphism.

  Denote by
  $\overline M$ the factor module $M/\mathrm{Im}(w)$ and by $\pi\colon
  M\to \overline M$ the canonical surjection. Let $[s,t]\colon M\to
  M_1\oplus M_2$ be the inverse morphism of $[u+v,w]^T$. This amounts
  to the following equalities
  \begin{enumerate}
  \item $s\cdot (u+v) +tw = \mathrm{Id}_M$,
  \item $\left(
      \begin{array}{cc}
        (u+v)s& (u+v) t \\
        ws & wt
      \end{array}\right) =
    \left(
      \begin{array}{cc}
        \mathrm{Id}_{M_1}& 0 \\
        0 & \mathrm{Id}_{M_2}
      \end{array}\right)$.
  \end{enumerate}
  In particular, $s\colon M\to M_1$ induces a morphism $\overline
  s\colon \overline M\to M_1$ such that $\pi \overline s = s$. It
  satisfies
  \begin{itemize}
  \item $\pi \overline s \cdot (u+v) \pi = s\cdot (u+v) \pi = (s\cdot
    (u+v) + tw)\pi  = \pi$,
  \item $(u+v)\pi \overline s = (u+v)s = \mathrm{Id}_{M_1}$.
  \end{itemize}
  Hence, $(u+v)\pi\colon M_1\to \overline M$ is an isomorphism with
  inverse $\overline s$. Since $M_1$ is indecomposable, at least one
  of the two morphisms $u\pi \colon M_1\to \overline M$ or $v\pi \colon
  M_1\to \overline M$ is an isomorphism. Assume the former, then
  $\overline M=\mathrm{Im}(u\pi)$. Accordingly, $M=\mathrm{Im}(u)+\mathrm{Im}(w)$. Therefore, $[u,w]^T\colon M\to M$ is a surjection, and
  hence an isomorphism.
\end{proof}

\subsection{Properties on kernels of irreducible morphisms}
\label{sec:prop-kern-irred}

The following lemma  compares the kernel of an irreducible morphism
with the kernel of the corestriction to a proper direct summand of its codomain.
The reader may formulate the dual version of this lemma with a
completely analogous proof.
\begin{lem}
\label{sec:prop-kern-irred-1}
  Let $f=[f_1,f_2]\colon X\to Y_1\oplus Y_2$ be an irreducible
  morphism where $X\in\mathrm{ind}\,A$ and $Y_1,Y_2\in \mathrm{mod}\,A$ are
  nonzero. Let $i\in\{1,2\}$.
 If $f$ is an epimorphism then
    \begin{enumerate}[(a)]
    \item $\mathrm{Ker}(f)\not\simeq \mathrm{Ker}(f_i)$,
    \item $\mathrm{Ker}(f)$ and $\mathrm{Ker}(f_i)$ are non injective,
    \item $\mathrm{Ker}(f_i)$ is non simple and the middle term of an
      almost split sequence starting at $\mathrm{Ker}(f_i)$ is indecomposable.
    \end{enumerate}
\end{lem}
\begin{proof}
  There is no loss of generality assuming that $i=1$.
  Note that $f_1$ is an epimorphism because so is $f$. There are
  short exact sequences $0\to \mathrm{Ker}(f)\to X\xrightarrow f
  Y_1\oplus Y_2\to 0$ and $0\to \mathrm{Ker}(f_1)\to X\xrightarrow{f_1}
  Y_1\to 0$ in $\mathrm{mod}\,A$. Since $\mathrm{dim}_{\k}Y_1<\mathrm{dim}_{\k}
  (Y_1\oplus Y_2)$, a length argument shows that $\mathrm{Ker}(f)\not\simeq
  \mathrm{Ker}(f_1)$. This proves $(a)$. Using the same exact sequences,
  the fact that $X\in \mathrm{ind}\,A$ entails that $\mathrm{Ker}(f)$ and
  $\mathrm{Ker}(f_1)$ are non injective. This proves $(b)$. In order to
  prove $(c)$ it suffices to prove that $\mathrm{Ker}(f_1)$ is
  non simple, thanks to \cite{Kra} (see also
  \cite{Bre}). Note that $\mathrm{Ker}(f)=\mathrm{Ker}(f_1)\cap \mathrm{Ker}(f_2)$ is a proper (and indecomposable) submodule of $\mathrm{Ker}(f_1)$. Hence $\mathrm{Ker}(f_1)$ is non simple. This proves
  $(c)$.
\end{proof}

The following lemma is  proved in \cite[Theorem 3.2]{CD} for algebras over algebraically
closed fields. The result given there still works for algebras over
artin rings. Its main argument is recalled below for the convenience of
the reader. As usual, given a morphism $X\to Y$ in $\mathrm{mod}\,A$, a \emph{kernel (morphism)} for $f$ is a morphism
$K\to X$ such that the induced sequence of functors $0\to \mathrm{Hom}_A(-,K)\to \mathrm{Hom}_A(-,X)\to \mathrm{Hom}_A(-,Y)$ is exact.

\begin{lem}
  \label{sec:appl-comp-irred-1}
  Assume that $A$ is an artin algebra over an artin ring.
  Let $0\to K\xrightarrow{i}X \xrightarrow{f} Y\to 0$ be an exact sequence in $\mathrm{mod}\,A$ such that $X$ is indecomposable and $f$ is irreducible. Assume that
  there exists $n\in \mathbb N$ such that $i\in \mathrm{rad}^n\backslash\mathrm{rad}^{n+1}$.
  \begin{enumerate}
  \item For any morphism $K\to X$ lying in $\mathrm{rad}^n\backslash\mathrm{rad}^{n+1}$, there exists an automorphism $X\to X$ such that the
    composition morphism $K\to X\to X$ is a kernel morphism of $f$.
  \item If $X$ is indecomposable, then there exists a path $X_0=K\to
    X_1\to \cdots \to X_{n-1}\to
    X_n=X$ of irreducible morphisms between indecomposables and with
    composition equal to a kernel morphism of $f$.
  \end{enumerate}
\end{lem}
\begin{proof}
  (1) This is obtained upon applying \cite[Proposition 5.7,
  p. 173]{ARS} to the given morphism $u\colon K\to X$. The cited
  result asserts that
  either $u$ factors through $i$, or $i$ factors through $u$.

  \medskip

  (2) The class of $i$ modulo $\mathrm{rad}^{n+1}$ is a non trivial sum of
  compositions of paths of length $n$. Since $i\not\in \mathrm{rad}^{n+1}$, at least one of these paths has composition not lying
  in $\mathrm{rad}^{n+1}$. Using (1), such a path fits the conclusion (up
  to composition of the last
  morphism of that path with an automorphism of $X$).
\end{proof}

\subsection{Categories with length}
\label{sec:categ-with-length}

Some proofs in this text make use of specific functors taking values
in $\mathrm{ind}\,A$. These functors are defined over categories with
length. This subsection defines these categories.

Let $\mathcal C$ be
a $\k$-linear category. Assume that $\mathrm{dim}_\k\mathcal
C(x,y)<\infty$ for every $x,y\in \mathcal C$. Assume also that
distinct objects in $\mathcal C$ are not isomorphic. Define $\mathcal
R\mathcal C$ to be the ideal of $\mathcal C$ consisting of those
morphisms that are not invertible. The powers of $\mathcal R\mathcal
C$ are defined recursively by $\mathcal R^0\mathcal C=\mathcal C$,
$\mathcal R^1\mathcal C=\mathcal{RC}$ and $\mathcal R^{\ell+1}
\mathcal C=\mathcal R^{\ell}\mathcal C\cdot
\mathcal{RC}=\mathcal{RC}\cdot \mathcal R^{\ell}\mathcal C$. A
morphism in $\mathcal C$ is called \emph{irreducible} if it lies in
$\mathcal{RC}$ and not in $\mathcal R^2\mathcal C$. A \emph{path} (of
irreducible morphisms) in $\mathcal C$ is a sequence
$x_0\xrightarrow{\varphi_1} x_1\to \cdots \to
x_{\ell-1}\xrightarrow{\varphi_n}x_n$ where $x_0,\ldots,x_n\in
\mathcal C$ and $\varphi_i$ is an irreducible morphism from $x_{i-1}$
to $x_i$ for every $i$; By definition the length of such a path is
$n$.

A category \emph{with length} is a $\k$-linear category $\mathcal C$
as above such that
\begin{enumerate}[(a)]
\item for every $x,y\in \mathcal C$, the paths of irreducible
  morphisms from $x$ to $y$ all have the same length,
\item $\bigcap_{n\geqslant 0}\mathcal R^n\mathcal C=0$ (in particular,
  any morphism in $\mathcal C$ is a sum of compositions of paths of
  irreducible morphisms in $\mathcal C$).
\end{enumerate}
Note that, for such a category, $\mathcal C(x,x)$ is a division
$\k$-algebra for every $x$.

As an example, the mesh category of any modulated
translation quiver (\cite[Section 1.2]{IT}) with length is a category
with length. Recall that
a quiver with length is a quiver such that any two parallel paths have
the same length.
The categories with length used in this text are all of this
shape. However, the presentation here sticks to the general setting in
order to avoid unnecessary technicalities.

The following proposition is used throughout the text. It follows from
the definitions.
\begin{prop}
  \label{prop:with_length}
  Let $\mathcal C$ be a $\k$-linear category with length and $x,y\in
  \mathcal C$. If there exists a path of irreducible morphisms from
  $x$ to $y$ in $\mathcal C$ and with length denoted by $\ell$, then
  \begin{enumerate}
  \item $\mathcal C(x,y)=\mathcal{RC}(x,y)=\cdots =\mathcal
    R^{\ell}\mathcal C(x,y)$,
  \item $\mathcal R^i\mathcal C(x,y)=0$ for every $i>\ell$,
  \item if there exists an irreducible morphism $x\to y$, then
    $\mathcal C(x,y)\backslash\{0\}$ consists of irreducible
    morphisms,
  \item for every $z\in \mathcal C$, if there exists a path of
    irreducible morphisms from $y$ to $z$ with length denoted by
    $\ell'$, then $\mathcal R^i\mathcal C(x,z)=0$, for every $i>\ell+\ell'$.
  \end{enumerate}
\end{prop}

\subsection{Functors with the covering property}
\label{sec:funct-with-cover}

Let $\Gamma$ be an Auslander-Reiten component of $A$. Let $\mathcal C$
be a $\k$-linear category with length and $F\colon \mathcal C\to \mathrm{ind}\,\Gamma$ be a $\k$-linear functor. This subsection introduces
the class of functors $F$ that are used in the proof of the main
results of this text.

\subsubsection{Definition and basic features}
\label{sec:defin-basic-feat}

By definition, $F$ is said to have the \emph{covering property} if it
satisfies the following conditions:
\begin{enumerate}[(a)]
\item for every $X\in \Gamma$, there exists $x\in \mathcal C$ such
  that $Fx=X$,
\item $F(\mathcal{RC})\subseteq \mathrm{rad}$,
\item for every $x,y\in \mathcal C$ and $n\geqslant 0$, the two following maps induced by
  $F$ are bijective
  \[
  \begin{array}{rclc}
    \bigoplus\limits_{Fz=Fy}\mathfrak{R}^n\k(\widetilde{\Gamma})(x,z)
    /
    \mathfrak{R}^{n+1}\k(\widetilde{\Gamma})(x,z)
    & \to &
            \mathrm{rad}^n(Fx,Fy)/\mathrm{rad}^{n+1}(Fx,Fy) \\
    \\
    \bigoplus\limits_{Fz=Fy}\mathfrak{R}^n\k(\widetilde{\Gamma})(z,x)/
    \mathfrak{R}^{n+1}\k(\widetilde{\Gamma})(z,x)
    & \to &
            \mathrm{rad}^n(Fy,Fx)/\mathrm{rad}^{n+1}(Fy,Fx), &
  \end{array}
  \]
\item for every $x,y\in \mathcal C$, the two following maps induced
  by $F$ are injective
  \[
  \bigoplus\limits_{Fz=Fy}\k(\widetilde{\Gamma})(x,z) \to
  \mathrm{Hom}_A(Fx,Fy)
  \ \ and\ \
  \bigoplus\limits_{Fz=Fy}\k(\widetilde{\Gamma})(z,x) \to
  \mathrm{Hom}_A(Fy,Fx),
  \]
\item for every $x\in \mathcal C$ and $Y\in \Gamma$, there exists at
  most one $y\in \mathcal C$ such that $Fy=Y$ and $\mathcal C(x,y)$ contains an
  irreducible morphism,
\item for every $X\in \Gamma$ and $y\in \mathcal C$, there exists at
  most one $x\in \mathcal C$ such that $Fx=X$ and $\mathcal C(x,y)$
  contains an irreducible morphism.
\end{enumerate}
When $F$ has the covering property, it induces a $\k$-algebra
isomorphism
\[
\mathcal C(x,x)/\mathcal{RC}(x,x)\xrightarrow{\sim}\mathrm{End}_A(Fx)/\mathrm{rad}(Fx,Fx)
\]
for every $x\in \mathcal C$.
Also it
induces a linear isomorphism
$\mathcal{RC}(x,y)/\mathcal R^2\mathcal C(x,y)\xrightarrow{\sim}\mathrm{irr}(Fx,Fy)$ for every $x,y\in \mathcal C$ such that there exists an
irreducible morphism $x\to y$ in $\mathcal C$.

Well-behaved functors in the sense of \cite{CLT1,CLT2} are examples
of functors with the covering property where $\mathcal C$ equals the
mesh category of the universal cover of the modulated translation
quiver $\Gamma$. The following result plays a central role in the
present text.
\begin{prop}[Proposition 2.5 of \cite{CLT2}]
  \label{subsec_wellbehavedexistence}
  Let $A$ be a finite dimensional algebra over a perfect field
  $\k$. Let $\Gamma$ be an Auslander-Reiten component of $A$.
  \begin{enumerate}
  \item There exists a $\k$-linear functor $F\colon \mathcal C\to
    \mathrm{ind}\,\Gamma$ with the covering property.
  \item Let $X\in \Gamma$ and let
    $f\colon X\to \oplus_{i=1}^rX_i^{n_i}$ be a freely  irreducible
    morphism where $X_1,\ldots,X_r$ are pairwise non isomorphic. There
    exists a $\k$-linear functor $F\colon \mathcal C\to \mathrm{ind}\,\Gamma$
    with the covering property, and there exist $x\in F^{-1}X$,
    $x_i\in F^{-1}X_i$ ($1\leqslant i\leqslant r$) such that each
    component $X\to X_i$ of $f$ lies in the image of the mapping
    $\mathcal C(x,x_i)\to \mathrm{Hom}_A(X,X_i)$ induced by $F$.
  \end{enumerate}
\end{prop}

\subsubsection{Homogeneous morphisms}
\label{sec:homog-morph-1}

The notion of homogeneous morphism that is considered in this text is
relative to a given functor  $F\colon \mathcal C\to \mathrm{ind}\,\Gamma$
with the covering property. It is used to express the assumptions in
the main results of this text.

Recall from the introduction that, a morphism $X\to Y$ in $\mathrm{mod}\,A$ is called homogeneous if
there exist tuples  $(x_s)_s$ and $(y_t)_t$ of objects in
$\mathcal C$ and a tuple of morphisms $(\varphi_{s,t})_{s,t}\in
\oplus_{s,t}\mathcal C(x_s,y_t)$ such that $X=\oplus_x Fx_s$,
such that $Y=\oplus_t Fy_t$ and such that, for every $s,t$, the component
$Fx_s\to Fy_t$
of the given morphism $X\to Y$ is equal to $F(\varphi_{s,t})$. In such a case, the
morphism is called homogeneous with respect to $(y_t)_t$ (or to
$(x_s)_s$) if it is necessary to work with a predetermined
decomposition $Y=\oplus_tFy_t$ (or $X=\oplus_sFx_s$,
respectively). Note that, if the given morphism $X\to Y$ lies in $\mathrm{rad}^m$ for some $m$, then it is necessary that $\varphi_{s,t}\in
\mathcal R^m\mathcal C$ for every $s,t$.

Given $d\in \mathbb
N$, a morphism $f\colon X\to Y$ is called homogeneous \emph{up to $\mathrm{rad}^{d+1}$} if it is the sum of two morphisms from $X$ to $Y$,
the former being homogeneous and the latter lying in $\mathrm{rad}^{d+1}$. If the former morphism is equal to
$[F(\varphi_{s,t})\ ;\ s,t]\colon X\to Y$ (with the previous notation),
then each $\varphi_{s,t}$ may be chosen to be zero when it lies in
$\mathcal R^{d+1}\mathcal C$. Note that this condition makes the
decomposition of $f$ unique. In such a case,
the morphism $[F(\varphi_{s,t})\ ;\ s,t]\colon X\to Y$ is called the
\emph{homogeneous part} of $f$.

For instance, for every $X\in \Gamma$, any irreducible morphism
$X\to \oplus_t Y_t$ is homogeneous up to $\mathrm{rad}^2$ (with respect
to any given $x\in \mathcal C$ such that $Fx=X$, and for any given
$F\colon \mathcal C\to \mathrm{ind}\,\Gamma$). Also, when this
irreducible morphism is freely irreducible, then there exists a
functor with the covering property for which that irreducible morphism
is homogeneous (see Proposition~\ref{subsec_wellbehavedexistence}).

The following lemma is useful in the proof of the main results of this
text. It deals with homogeneous monomorphisms. Consider the following
setting. Let $i\colon K\to X$ be a
monomorphism in $\mathrm{mod}\,A$ where every indecomposable direct summand
of $X$ lies in $\Gamma$. Let $X=\oplus_sX_s$ be a direct sum
decomposition into indecomposable modules (with $X_s\in \Gamma$, for
every $s$) and, for every $s$, let $x_s\in \mathcal C$ be such that
$Fx_s=X_s$. Fix a direct sum decomposition
$K=K^{(\infty)}\oplus \left( \oplus_{m\geqslant
    0}\oplus_rK^{(m)}_r\right)$ such that each $K^{(m)}_r$ is
indecomposable, such that the restriction $K^{(\infty)}\to X$ of $i$
lies in $\mathrm{rad}^\infty$ and such that each restriction
$K^{(m)}_r\to X$ of $i$ lies in $\mathrm{rad}^m\backslash\mathrm{rad}^{m+1}$. For simplicity, $\oplus_rK^{(m)}_r$ is denoted by
$K^{(m)}$ for every $m$.
\begin{lem}
  \label{sec:funct-with-cover-1}
  Under the previous setting, assume that the restriction
  $\oplus_{m,r}K^{(m)}_r\to X$ of $i$ is homogeneous with
  respect to $(x_s)_s$.
  Then, for every
  $\ell\geqslant n$, $Z\in \Gamma$ and $h\in \mathrm{Hom}_A(Z,K)$, the
  following conditions are equivalent
  \begin{enumerate}[(i)]
  \item $hi\in \mathrm{rad}^{\ell}$,
  \item for every $m$, the component $Z\to K^{(m)}$ of $h$ lies in
    $\mathrm{rad}^{\ell-m}$.
  \end{enumerate}
\end{lem}
\begin{proof}
  Obviously, $(ii)\implies (i)$. Assume $(i)$. There is no loss of
  generality in assuming that $K^{(\infty)}=0$. Since $i$ is
  homogeneous, there exists $k^{(m)}_r\in \mathcal C$ such that
  $Fk^{(m)}_r=K^{(m)}_r$, for every $(m,r)$, and there exists, for
  every $m,r,s$, a morphism $\iota^{(m)}_{r,s}\in \mathcal R^m\mathcal
  C(k^{(m)}_r,x_s)$
that is zero when it lies in $\mathrm{rad}^{m+1}$ and such that
$i=[F(\iota^{(m)}_{r,s}) \ ;\ m,r,s]$.
For every $m,r$, denote by
  $h_r^{(m)}\colon Z\to K_r^{(m)}$ the corresponding component of
  $h$; There exists $(\eta_{z,r}^{(m)})_{z\in
    F^{-1}Z}\in \oplus_z\mathcal C(z,k_r^{(m)})$ such that
  \begin{itemize}
  \item $h_r^{(m)}-\sum_zF(\eta^{(m)}_{z,r})\in \mathrm{rad}^{\ell-m}$,
  \item for every $z$, the morphism $\eta^{(m)}_{z,r}$ is $0$ if it
    lies in $\mathcal R^{\ell-m}\mathcal C$.
  \end{itemize}
  Then $(h_r^{(m)}-\sum_zF(\eta^{(m)}_{z,r}))F(\iota^{(m)}_{r,s})$
  lies in $\mathrm{rad}^{\ell}$ and equals
  $h_r^{(m)}F(\iota^{(m)}_{r,s})-
  \sum_zF(\eta^{(m)}_{z,r}\iota^{(m)}_{r,s})$.
  Since $hi$ lies $\mathrm{rad}^\ell$, then so does
  $\sum_{r,m}h_r^{(m)}F(\iota^{(m)}_{r,s})$ for every $s$. Accordingly,
  $\sum_zF\left(
    \sum_{r,m}\eta_{z,r}^{(m)}\iota^{(m)}_{r,s}\right)$ lies in $\mathrm{rad}^\ell$ for every $s$, and hence
  $\sum_{r,m}\eta_{z,r}^{(m)}\iota^{(m)}_{r,s}$ lies in $\mathcal
  R^\ell\mathcal C(z,x_s)$ for every $z,s$.

  Now, for every $r,m,z,s$, if both $\eta^{(m)}_{z,r}$ and
  $\iota^{(m)}_{r,s}$ are nonzero, then there exists a path of
  irreducible morphisms of length at most $\ell-1$ from $z$ to $x_s$
  in $\mathcal C$, and hence $\mathcal R^\ell\mathcal
  C(z,x_s)=0$. Thus
  $\sum_{r,m}\eta^{(m)}_{z,r}\iota^{(m)}_{r,s}=0$ for every $z,s$.
  Applying $F$ to this equality yields that the following composite
  morphism is zero
  \[
  Z\xrightarrow{
    [ \sum_z F(\eta^{(m)}_{z,r})\ ;\ r,m]
  }
  \oplus_{r,m}K^{(m)}_r \xrightarrow{i} X\,.
  \]
  Given that $i$ is a monomorphism, it follows that $\sum_z
  F(\eta^{(m)}_{z,r})=0$ for every $m,r$, and hence
  $\eta^{(m)}_{z,r}=0$ for every $m,r,z$. Accordingly, $h^{(m)}_r\in
  \mathrm{rad}^{\ell-m}$ for every $m,r$.
\end{proof}

\section{Investigation of the left degree}
\label{sec:degr-homog-morph}

This section investigates the left degree of a morphism $f\colon X\to
Y$ under certain conditions expressed in terms of functors with the
covering property. The purpose of the section is to prove
Proposition~\ref{sec:kern-char-homog-1}: Assuming that $f$ lies in
$\mathrm{rad}^d\backslash\mathrm{rad}^{d+1}$ and is homogeneous up to
$\mathrm{rad}^d$ for some $d\in \mathbb N$ and that $f$ has
finite left degree denoted by $n$, there exists a direct sum
decomposition $\mathrm{Ker}(f') = K^{(\infty)} \oplus\left(
  \oplus_{m\geqslant n}\oplus_rK^{(m)}_r\right)$ of the kernel of the
homogeneous part $f'$ of $f$ such as in the setting of
Lemma~\ref{sec:funct-with-cover-1} and such that for every $Z\in \mathrm{ind}\,A$ and every $\ell\geqslant m$, the
following sequence is exact
\[
0\to \bigoplus\limits_{n\leqslant m \leqslant \ell} \frac{\mathrm{rad}^{\ell-m}}{\mathrm{rad}^{\ell-m+1}}(Z,K^{(m)})
\to
\frac{\mathrm{rad}^{\ell}}{\mathrm{rad}^{\ell+1}}(Z,X)
\to
\frac{\mathrm{rad}^{\ell+d}}{\mathrm{rad}^{\ell+d+1}}(Z,Y)\,.
\]

The general setting in which this result is valid is presented in
\ref{sec:setting-study} and is assumed throughout the section. Next,
a key lemma is proved in \ref{sec:homog-morph-2}. Then
\ref{sec:gener-kern-char} is devoted to the construction of
the above mentioned direct sum decomposition of $\mathrm{Ker}(f')$. Finally, \ref{sec:kern-char-homog} is devoted to
Proposition~\ref{sec:kern-char-homog-1}.

\subsection{Setting of the study}
\label{sec:setting-study}

Given $f \colon  X \to Y$ lying in $\mathrm{rad}^d\backslash\mathrm{rad}^{d+1}$ for
some integer $d\geqslant 1$, recall that $d_\ell(f)=n$
if $n$ is the least integer such that there exists $Z \in \mathrm{ind}\,A$  and $g \in \mathrm{Hom}_A (Z,X)$ verifying
$g \in \mathrm{rad}^n\backslash\mathrm{rad}^{n+1}$ and $gf \in \mathrm{rad}^{n+d+1}$.

Let $\Gamma$ be an Auslander-Reiten component of $A$.  Let $F\colon
\mathcal C\to \mathrm{ind}\,\Gamma$ be a
functor with the covering property for
$\Gamma$. Let $f\colon X\to Y$ be a morphism in $\mathrm{mod}\,A$
satisfying the following conditions
\begin{itemize}
\item every indecomposable direct summand of $X\oplus Y$ lies in
  $\Gamma$,
\item $f$ lies in $\mathrm{rad}^d$ and is homogeneous up to $\mathrm{rad}^{d+1}$ for some $d\in \mathbb N$,
\item the left degree of $f$ is finite and denoted by $n$.
\end{itemize}

Let $X=\oplus_s X_s$ and $Y=\oplus_t Y_t$ be direct sum decompositions
such that there exist tuples of objects $(x_s)_s$ and $(y_t)_t$ in
$\mathcal C$ and a tuple of morphisms $(\varphi_{s,t})_{s,t}\in
\oplus_{s,t}\mathcal R^d\mathcal C(x_s,y_t)$ verifying
\begin{itemize}
\item $X_s=Fx_s$ and $Y_t=Fy_t$ for every $s,t$,
\item for every $s,t$, the morphism $\varphi_{s,t}$ is zero if it lies
  in $\mathcal R^{d+1}\mathcal C$,
\item $f-[F(\varphi_{s,t})\ ;\ s,t]$ lies in $\mathrm{rad}^{d+1}$.
\end{itemize}
Recall that the morphism $[F(\varphi_{s,t})\ ;\ s,t]$ is called the
homogeneous part of $f$. In what follows it is denoted by $f'$.
Recall also that the homogeneity condition made on $f$ is valid in any of
the following cases
\begin{itemize}
\item $f$ is irreducible (see \ref{sec:defin-basic-feat}),
\item for every $s,t$, the component $X_s\to Y_t$ of $f$ is the sum of
  a morphism in $\mathrm{rad}^{d+1}$ and of
  the composition of a path of irreducible morphisms between
  indecomposables and with length $d$,
\item for every $s,t$, the vector space $\mathrm{Hom}_A(X_s,Y_t)/\mathrm{rad}^\infty(X_s,Y_t)$ has dimension at most $1$.
\end{itemize}

The objective of the section is to establish  some general properties
on the left degree of $f$. These are expressed in terms of $f'$.

\subsection{First interpretation of the finiteness of left degree}
\label{sec:homog-morph-2}

The following lemma investigates the shape of a morphism lying in $\mathrm{rad}^\ell\backslash\mathrm{rad}^{\ell+1}$ and having composition with $f$
lying in $\mathrm{rad}^{\ell+d+1}$.

\begin{lem}
  \label{sec:gener-kern-char-1}
  Let $Z\in \Gamma$, $g\in \mathrm{Hom}_A(Z,X)$ and
  $\ell\geqslant n$  be such that
  $g\in \mathrm{rad}^\ell\backslash\mathrm{rad}^{\ell+1}$ and $gf\in \mathrm{rad}^{\ell+d+1}$. For every $s$, let $(\gamma_{z,s})_{z\in
    F^{-1}Z}\in \oplus_z \mathcal R^{\ell}\mathcal C(z,x_s)$ be such
  that $g$ is the sum of $g':=[\sum_zF(\gamma_{z,s})\ ;\ s]\colon Z\to
  X$ and a morphism in $\mathrm{rad}^{\ell+1}$ and such that, for every
  $z,s$, the morphism $\gamma_{z,s}$ is zero when it lies in $\mathcal
  R^{\ell+1}\mathcal C$. Then
  \begin{enumerate}[(a)]
  \item $\sum_s\gamma_{z,s}\varphi_{s,t}=0$ in $\mathrm{Hom}_A(Z,Y_t)$
    for every $z,t$, and hence
  \item   $g'f'=0$.
  \end{enumerate}
\end{lem}
\begin{proof}
It follows from the hypotheses that
  $g' f' - gf \in \mathrm{rad}^{\ell+d+1}$. Since,
  moreover,
  $ \left[\sum_zF(\gamma_{z,s})\ ;\
    s\right] f' = \left[\sum_zF(\sum_s\gamma_{z,s}\varphi_{s,t})\ ;\
    t\right]$,
  then $\sum_s\gamma_{z,s}\varphi_{s,t}\in \mathcal R^{\ell+d+1}\mathcal C$
  for every $z,t$. Given $s,t,z$, the morphism
  $\gamma_{z,s}\varphi_{s,t}$ is a linear combination of
  compositions of paths of irreducible morphisms of length $\ell+d$ in
  $\mathcal C$; Thus
  $\sum_s\gamma_{z,s}\varphi_{s,t}=0$ for every $z,t$ (see
  Proposition~\ref{prop:with_length}). This
  proves (a). And (b) is a direct consequence of (a).
\end{proof}

\subsection{Decomposition of the kernel}
\label{sec:gener-kern-char}

This subsection collects technical properties on the inclusion
morphism $i\colon \mathrm{Ker}(f')\to X$. Its aim is to show the
existence of a direct sum decomposition of $\mathrm{Ker}(f')$ fitting the
setting of Lemma~\ref{sec:funct-with-cover-1}.  First, in
\ref{sec:decomp-kern-rm}, it is proved that
$i\in \mathrm{rad}^n\backslash\mathrm{rad}^{n+1}$. From this property is
derived a first direct sum decomposition of $\mathrm{Ker}(f')$ that is
close to the desired one. After introducing useful notation on $i$ in
\ref{sec:notat-kern-morph}, the needed decomposition of
$\mathrm{Ker}(f')$ is then obtained in \ref{sec:reduct-incl-morph} by
taking the direct image of the one in \ref{sec:decomp-kern-rm} under a
suitable automorphism of $\mathrm{Ker}(f')$. Finally, it is proved in
\ref{sec:reduc-vers-kern} that $i$ is homogeneous with respect to the
resulting decomposition.

\subsubsection{A first decomposition of the kernel $\mathrm{Ker}(f')$}
\label{sec:decomp-kern-rm}

 The following lemma gives a first property on the inclusion
morphism $i\colon \mathrm{Ker}(f')\to X$.

\begin{lem}
  \label{sec:decomp-kern-rm-1}
  The inclusion morphism $\mathrm{Ker}(f')\to X$ lies in $\mathrm{rad}^n\backslash\mathrm{rad}^{n+1}$.
\end{lem}
\begin{proof}
  Since $f$ has finite left degree, there exists $Z,g$ fitting the
  hypotheses of Lemma~\ref{sec:gener-kern-char-1} with $\ell=n$. Using the notation
  introduced there, the morphism $[\sum_zF(\gamma_{z,s})\ ;\ s]\colon
  Z\to X$ lies in $\mathrm{rad}^n\backslash\mathrm{rad}^{n+1}$ and factors through $i$. Therefore
  $i\not\in \mathrm{rad}^{n+1}$. Moreover, $i\in \mathrm{rad}^n$ because
  $if'=0$ and $d_{\ell}(f')=n$.
\end{proof}

Hence, there exists a direct sum decompositions
\begin{equation}
  \label{eq:1}
  \left\{
    \begin{array}{l}
      \mathrm{Ker}(f') = K^{(\infty)} \oplus \left( \oplus_{n\leqslant m}
      K^{(m)} \right) \\
      K^{(m)}=\oplus_r K^{(m)}_r\ \text{for every $m$}
      \end{array}\right.
\end{equation}
such that
\begin{enumerate}[(a)]
\item each $K^{(m)}_r$ is indecomposable,
\item the inclusion morphism $K^{(\infty)}\to X$ lies in $\mathrm{rad}^\infty$,
\item for every $m,r$, the inclusion morphism $K^{(m)}_r\to X$ lies in
  $\mathrm{rad}^m\backslash\mathrm{rad}^{m+1}$.
\end{enumerate}
For every such set of decompositions (\ref{eq:1}) satisfying (a), (b)
and (c), denote by $\alpha_K$ the sequence indexed by the integers not
smaller than $n$ and with value at any $m\geqslant n$ equal to the
number of terms in the direct sum decomposition $K^{(m)}=\oplus_r
K^{(m)}_r$. In what follows, it is assumed that
\begin{enumerate}[(a)]
  \setcounter{enumi}{3}
\item the decomposition (\ref{eq:1}) satisfying (a), (b) and (c) is
  such that the sequence $\alpha_K$ is minimal for the lexicographic order.
\end{enumerate}

In the sequel,
the inclusion morphisms $K^{(m)}_r\to X$ and $K^{(m)}\to X$ are denoted
by $i^{(m)}_r$ and $i^{(m)}$, respectively.

\subsubsection{Notation for the inclusion morphism $\mathrm{Ker}(f')\to
  X$}
\label{sec:notat-kern-morph}

Recall that  $i$ denotes the inclusion morphism $\mathrm{Ker}(f')\to X$.
In order to prove the main result of this section, it is useful to
show the existence of an automorphism $\sigma$ of $\mathrm{Ker}(f')$ such
that each restriction $K^{(m)}\to X$ of $\sigma i$ is homogeneous with
respect to $(x_s)_s$. For that purpose, some notation is introduced below.
For every $m,r,s$, there exists
$(\iota^{(m)}_{z,r,s})_{z\in F^{-1}K^{(m)}_r}\in \oplus_z\mathcal
R^m\mathcal C(z,x_s)$ such that
\begin{itemize}
\item $i_r^{(m)}-[\sum_z F(\iota^{(m)}_{z,r,s})\ ;\ s]$ lies in
  $\mathrm{rad}^{m+1}$ (recall that $X=\oplus_sX_s$),
\item $\iota^{(m)}_{z,r,s}$ is zero whenever it lies in $\mathcal
  R^{m+1}\mathcal C$.
\end{itemize}

The following lemma collects useful factorisation properties of these morphisms.
\begin{lem}
  \label{sec:notat-kern-morph-2}
  The families of morphisms $\{i_r^{(m)}\}_{r,m}$ and
  $\{\iota^{(m)}_{z,r,s}\}_{m,z,r,s}$ have the following properties.
  \begin{enumerate}
  \item For every $z,m,r$ (with $Fz=K^{(m)}_r$), there exists
    $\sigma^{(m)}_{z,r}\colon K^{(m)}_r\to \mathrm{Ker}(f')$ such that
    $[F(\iota^{(m)}_{z,r,s})\ ;\ s] = \sigma_{z,r}^{(m)}i$.
  \item For every $(m,r)$, there exists $\tau^{(m)}_r\colon
    K^{(m)}_r\to \mathrm{Ker}(f')$ such that
    $i^{(m)}_r - [\sum_zF(\iota^{(m)}_{z,r,s})\ ;\ s] = \tau^{(m)}_r
    i$.
  \item Moreover, $\tau^{(m)}_r+\sum_z \sigma^{(m)}_{z,r}$ is the
    inclusion morphism $K^{(m)}_r\to \mathrm{Ker}(f')$.
  \end{enumerate}
\end{lem}
\begin{proof}
  (1) Fix $m$ and $r$. Since $i^{(m)}_rf'=0$ and
  $f'=[F(\varphi_{s,t})\ ;\ s,t]$, the morphism
  $[\sum_zF(\sum_s\iota^{(m)}_{z,r,s}\varphi_{s,t})\ ;\ t]$ from
  $K_r^{(m)}$ to $Y$ lies in $\mathrm{rad}^{m+d+1}$. In view of
  Propositions~\ref{prop:with_length} and
  \ref{subsec_wellbehavedexistence}, it follows that
  $\sum_s\iota^{(m)}_{z,r,s}\varphi_{s,t}$ lies in
  $\mathcal R^{m+d+1}\mathcal C$ for every $z,s,t$, and hence is zero
  (note that $\iota^{(m)}_{z,r,s}\varphi_{s,t}$ is a linear
  combination of compositions of paths of irreducible morphisms with
  length $m+d$). Thus $[F(\iota^{(m)}_{z,r,s})\ ;\ s]f'$ is zero, and
  therefore $[F(\iota^{(m)}_{z,r,s})\ ;\ s]$ factors through $i$ for
  every $z$.

  \medskip

  (2) Note that $\left(
    i^{(m)}_r - [\sum_zF(\iota^{(m)}_{z,r,s})\ ;\ s]
  \right) f'$ is zero because so are $i^{(m)}_rf'$  and
  $[F(\iota^{(m)}_{z,r,s})\ ;\ s]f'$ (for every $z$). Hence, there
  exists $\tau^{(m)}_r\colon K^{(m)}_r\to \mathrm{Ker}(f')$ such that
  $i^{(m)}_r - [\sum_zF(\iota^{(m)}_{z,r,s})\ ;\ s] = \tau^{(m)}_r
  i$.

  \medskip

  (3) According to (1) and (2), the composition $(\tau_r^{(m)} +
  \sum_z \sigma^{(m)}_{z,r}) i$ equals $i^{(m)}_r$. Whence the
  conclusion.
\end{proof}

Note that $\tau^{(m)}_r i = \left(\sum_z\sigma_{z,r}^{(m)}\right) i -i^{(m)}_r\in
\mathrm{rad}^{m+1}$ for every $m$.

\subsubsection{The direct sum decomposition of $\mathrm{Ker}(f')$}
\label{sec:reduct-incl-morph}

The following lemma shows the existence of an automorphism $\sigma$ of
$\mathrm{Ker}(f')$ such that each restriction $K^{(m)}_r\to X$ of $\sigma
i$ is homogeneous with respect to $(x_s)_s$.
\begin{lem}
  \label{sec:reduct-incl-morph-1}
  Let $i\colon \mathrm{Ker}(f')\to X$ be the inclusion morphism. There
  exists an automorphism $\sigma$ of $\mathrm{Ker}(f')$
  and, for every $(m,r)$, there exist
  \begin{itemize}
  \item $k_r^{(m)}\in F^{-1}K^{(m)}_r$,
  \item $\iota^{(m)}_{r,s}\in \mathcal R^m\mathcal C(k^{(m)}_r,x_s)$
    (for every $s$) not lying in $\mathcal R^{m+1}\mathcal C\backslash\{0\}$,
  \end{itemize}
  such that the restriction of $\sigma i\colon \mathrm{Ker}(f')\to X$ to
  $K^{(m)}_r$ equals $[F(\iota^{(m)}_{r,s})\ ;\ s]$.  In particular,
  each restriction $K^{(m)}_r\to X$ of $\sigma i$ is homogeneous with
  respect to $(x_s)_s$.
\end{lem}
\begin{proof}
  Apply Lemma~\ref{sec:making-section-from-2} to the following data
  \begin{itemize}
  \item take $M$ to be $\mathrm{Ker}(f')$,
  \item take $M=\oplus_i M_i$ to be a direct sum decomposition into
    indecomposables subordinated to the decomposition
    $\mathrm{Ker}(f')=K^{(\infty)}\oplus\left( \oplus_{m\geqslant
        n}\oplus_r K^{(m)}_r\right)$,
  \item for every $i$ such that $M_i$ is a summand of $K^{(\infty)}$,
    take $\{\lambda_{i,j}\}_j$ to be the family consisting of a single
    element, namely the inclusion morphism $M_i\to \mathrm{Ker}(f')$,
  \item for every $i$ such that $M_i=K^{(m)}_r$ for some $(m,r)$, take
    $\{\lambda_{i,j}\}_j$ to be the family consisting of the morphisms
    $\sigma^{(m)}_{z,r}$ (for $z\in F^{-1}K^{(m)}_r$) together with
    the morphism $\tau^{(m)}_r$ (see
    Lemma~\ref{sec:notat-kern-morph-2}).
  \end{itemize}
  According to Lemma~\ref{sec:making-section-from-2}, there exists an
  automorphism $\sigma\colon \mathrm{Ker}(f')\to \mathrm{Ker}(f')$ such that
  \begin{itemize}
  \item its restriction to $K^{(\infty)}$ is the inclusion morphism,
  \item for every $(m,r)$, either there exists $k^{(m)}_r\in F^{-1}K^{(m)}_r$
    such that the restriction of $\sigma$ to $K^{(m)}_r$ is
    $\sigma^{(m)}_{k^{(m)}_r,r}$,
    or else the restriction of $\sigma$ to $K^{(m)}_r$ is
    $\tau^{(m)}_r$; In the former case, the composite morphism
    $K^{(m)}_r \xrightarrow{\sigma_{|K^{(m)}_r}} \mathrm{Ker}(f')\xrightarrow i X$ equals
    $[F(\iota^{(m)}_{k^{(m)}_r,r,s})\ ;\ s]$ and lies in $\mathrm{rad}^m$. In the latter case, it equals $i^{(m)}_r-[\sum_{Fz =
      K^{(m)}_r}F(\iota^{(m)}_{z,r,s})\ ;\ s]$ and lies in $\mathrm{rad}^{m+1}$. In view of the minimality condition (d) assumed on
    the set of decompositions (\ref{eq:1}) (see
    \ref{sec:decomp-kern-rm}), it is necessary that the former case
    always occurs and never does the latter case.
  \end{itemize}
  Now, denote by $\sigma^{(m)}_r$ and  $\iota^{(m)}_{r,s}$ the
  morphisms $\sigma^{(m)}_{k^{(m)}_r,r}$ and
  $\iota^{(m)}_{k^{(m)}_r,r,s}$, respectively. Thus
  $\sigma^{(m)}_r i = [F(\iota^{(m)}_{r,s})\ ;\ s]$. Note that, since
  $\iota^{(m)}_{r,s}$ is zero whenever it lies in $\mathcal
  R^{m+1}\mathcal C$ and since $\sigma i$ is a monomorphism, it is
  necessary that $\iota^{(m)}_{r,s}\not\in \mathcal R^{m+1}\mathcal
  C$.  The conclusion of
  the lemma therefore follows from the previous considerations and
  from those in \ref{sec:notat-kern-morph}.
\end{proof}

\subsubsection{Homogeneity of the inclusion morphism $\mathrm{Ker}(f')\to
X$}
\label{sec:reduc-vers-kern}

Replacing the set of decompositions (\ref{eq:1}) by their respective
direct images under any automorphism $\sigma$ of $\mathrm{Ker}(f')$ such
as in Lemma~\ref{sec:reduct-incl-morph-1} yields the following result.

\begin{prop}
  \label{sec:reduc-vers-kern-1}
  There exist direct sum decompositions
  \[
  \left\{
    \begin{array}{l}
      \mathrm{Ker}(f') = K^{(\infty)} \oplus \left( \oplus_{n\leqslant m}
      K^{(m)} \right) \\
      K^{(m)}=\oplus_r K^{(m)}_r\ \text{for every $m$}
      \end{array}\right.
  \]
  such that
  \begin{enumerate}
  \item each $K^{(m)}_r$ is indecomposable,
  \item the inclusion morphism $K^{(\infty)}\to X$ lies in $\mathrm{rad}^\infty$,
  \item for every $m,r$, the inclusion morphism $K^{(m)}_r\to X$
    lies in $\mathrm{rad}^m\backslash\mathrm{rad}^{m+1}$ and is homogeneous
    with respect to $(x_s)_s$.
  \end{enumerate}
\end{prop}
\begin{proof}
  This follows readily from Lemma~\ref{sec:reduct-incl-morph-1} after
  taking into account the comment that precedes the statement of the proposition.
\end{proof}

The following consequence of the previous result is useful in the
proof of the main results of this text.

\begin{cor}
  \label{sec:reduc-vers-kern-2}
  Let $i\colon \mathrm{Ker}(f')\to X$ be the inclusion morphism.
  Let $(K^{(m)}_r)_{m,r}$ be such as in Proposition~\ref{sec:reduc-vers-kern-1}.
  For every $\ell\geqslant n$, $Z\in
  \Gamma$ and $h\in \mathrm{Hom}_A(Z,K)$ the following conditions are
  equivalent
  \begin{enumerate}[(i)]
  \item $hi\in \mathrm{rad}^\ell$,
  \item for every $m,r$, the component $Z\to K^{(m)}_r$ of $h$ lies in
    $\mathrm{rad}^{\ell-m}$.
  \end{enumerate}
\end{cor}
\begin{proof}
  This follows from Lemma~\ref{sec:funct-with-cover-1} and
  Proposition~\ref{sec:reduc-vers-kern-1}.
\end{proof}

\subsection{The kernel characterisation for homogeneous morphisms}
\label{sec:kern-char-homog}

It is now possible to prove the main result of this section. It
translates the finiteness of the left degree of $f$ in terms of exact
sequences. As a reminder, here is the setting under which this result
is valid: $\Gamma$ is an Auslander-Reiten component of $A$ and
$F\colon \mathcal C\to \mathrm{ind}\,\Gamma$ is a functor with the covering
property relatively to $\Gamma$; a morphism $f\colon X\to Y$ is given
such that every indecomposable direct summand of $X\oplus Y$ lies in
$\Gamma$; it is assumed that $f$ lies in $\mathrm{rad}^d\backslash\mathrm{rad}^{d+1}$ and is homogeneous up to $\mathrm{rad}^{d+1}$ for some
$d\in\mathbb N$, and that $f$ has finite left degree denoted by $n$;
finally, the homogeneous part of $f$ is denoted by $f'$, the inclusion
morphism $\mathrm{Ker}(f')\to X$ is denoted by $i$ and a direct sum
decomposition of $\mathrm{Ker}(f')$ is given such as in
Proposition~\ref{sec:reduc-vers-kern-1}. For every $m$, the inclusion
morphism $K^{(m)}\to X$ is denoted by $i^{(m)}$.
\begin{prop}
  \label{sec:kern-char-homog-1}
  For every $Z\in \mathrm{ind}\,A$ and every integer $\ell\geqslant n$,
  the following sequence is exact
    \begin{equation}
      \label{eq:4}
      0\to \bigoplus\limits_{n\leqslant m \leqslant \ell} \frac{\mathrm{rad}^{\ell-m}}{\mathrm{rad}^{\ell-m+1}}(Z,K^{(m)})
      \to
      \frac{\mathrm{rad}^{\ell}}{\mathrm{rad}^{\ell+1}}(Z,X)
      \to
      \frac{\mathrm{rad}^{\ell+d}}{\mathrm{rad}^{\ell+d+1}}(Z,Y)\,.
      \tag{$\mathcal E_{Z,\ell}$}
    \end{equation}
\end{prop}
\begin{proof}
  Note that the arrows of the sequence are induced by the restrictions
  $i^{(m)}\colon K^{(m)}\to X$ of
  $i$ and by
  $f\colon X\to Y$. The exactness at $\oplus_{n\leqslant m\leqslant
    \ell} \frac{\mathrm{rad}^{\ell-m}}{\mathrm{rad}^{\ell-m+1}}(Z,K^{(m)})$
  follows from Corollary~\ref{sec:reduc-vers-kern-2}. To prove the
  exactness at $\frac{\mathrm{rad}^{\ell}}{\mathrm{rad}^{\ell+1}}(Z,X)$,
  consider $g\in \mathrm{rad}^\ell(Z,X)$ such that $gf\in \mathrm{rad}^{\ell+d+1}(Z,Y)$.

  Following
  Lemma~\ref{sec:gener-kern-char-1}, there exists $g'\in \mathrm{rad}^{\ell}(Z,X)$ such that $g-g'\in \mathrm{rad}^{\ell+1}$ and
  $g'f'=0$. Note that, following
  Proposition~\ref{sec:reduc-vers-kern-1}, the inclusion morphism
  $\mathrm{Ker}(f')\to X$ may be written as $[i^{(m)}\ ;\
  n\leqslant m\leqslant \ell]^T+a$ where  $a\in
  \mathrm{rad}^{\ell+1}$. Accordingly, there exists
  $(h^{(m)})_{n\leqslant m\leqslant \ell}\in \oplus_m\mathrm{Hom}_A
  (Z,K^{(m)})$ such that $g'-\sum_mh^{(m)}i^{(m)}\in \mathrm{rad}^{\ell+1}$.

  Now, Corollary~\ref{sec:reduc-vers-kern-2} entails that $h^{(m)}\in
  \mathrm{rad}^{\ell-m}$ because $g'\in \mathrm{rad}^\ell$. Denote by
  $\overline g$ and $\overline{h^{(m)}}$ the classes of
  $g$ and $h^{(m)}$ in $\frac{\mathrm{rad}^\ell}{\mathrm{rad}^{\ell+1}}(Z,X)$ and $\frac{\mathrm{rad}^{\ell-m}}{\mathrm{rad}^{\ell-m+1}}(Z,K^{(m)})$, respectively. Taking into account that
  $i^{(m)}\in \mathrm{rad}^m$ for every $m$, it follows from the previous
  considerations that $\overline g$ is the image of
  $(\overline{h^{(m)}})_{n\leqslant m\leqslant \ell}$, which finishes
  proving that (\ref{eq:4}) is exact.
\end{proof}

\section{Applications  to irreducible morphisms}
\label{sec:applications1}

This section applies the considerations of the previous section to  some
structure results on irreducible morphisms. In \ref{sec:general-case},
the main result of the previous section is applied to irreducible
morphisms $X\to Y$ where $X$ or $Y$ is indecomposable. And some
consequences are derived. Next, in \ref{subsec_degreealmostsplit}, the
kernels of irreducible morphisms $X\to Y$ are compared when their left
degree is finite. Finally, these results are applied in
\ref{subsec_finitetype} to characterise algebras of finite
representation type in terms of left degrees of irreducible morphisms.
The results of this section extend the results proved in \cite{CLT1}
to algebras over perfect fields and also strengthens them.

\subsection{Irreducible morphisms with finite (left) degree}
\label{sec:general-case}

Since the kernel of an irreducible morphism is indecomposable,
Proposition~\ref{sec:kern-char-homog-1} specialises to irreducible
morphisms as follows.

\begin{thm}
  \label{sec:finite-left-degree-1}
  Let $f\colon X\to Y$ be an irreducible morphism
  such that $X$ or $Y$ is indecomposable. Assume that $d_\ell(f)$ is
  finite and denoted by $n$. Then, there exists an irreducible
  morphism $f'\colon X\to Y$ such that
  \begin{enumerate}
  \item $f-f'\in \mathrm{rad}^2$,
  \item the inclusion morphism $i\colon \mathrm{Ker}(f') \to X$ lies in
    $\mathrm{rad}^n\backslash\mathrm{rad}^{n+1}$ and is the composition
    of $n$ irreducible morphisms between indecomposables,
  \item for every $\ell\geqslant n$, $Z\in \mathrm{ind}\,A$, the following sequence is exact
    \[
    0 \to
    \frac{\mathrm{rad}^{\ell-n}}{\mathrm{rad}^{\ell-n+1}}(Z,\mathrm{Ker}(f'))
    \xrightarrow{-\cdot i }
    \frac{\mathrm{rad}^\ell}{\mathrm{rad}^{\ell+1}}(Z,X)
    \xrightarrow{-\cdot f}
    \frac{\mathrm{rad}^{\ell+1}}{\mathrm{rad}^{\ell+2}}(Z,Y)\,.
    \]
\end{enumerate}
  If $f$ is freely irreducible, then $f'$ may be chosen
  equal to $f$.
\end{thm}
\begin{proof}
  Assume that $X$ is indecomposable (the remaining case is dealt with
  using dual considerations).
  Let $\Gamma$ be the Auslander-Reiten component containing $X$.
  Let $F\colon \mathcal C\to \mathrm{ind}\,\Gamma$ be a functor with the
  covering property for which $f$ is homogeneous up to $\mathrm{rad}^2$
  (see Proposition~\ref{subsec_wellbehavedexistence}). Note that, when
  $f$ is freely irreducible, it may be assumed to be
  homogeneous. Let $f'$ be the homogeneous part of $f$. In particular,
  $f=f'$ if $f$ is freely irreducible.  The conclusion
  therefore follows from Lemma~\ref{sec:appl-comp-irred-1} and
  Proposition~\ref{sec:kern-char-homog-1}.
\end{proof}

The previous result provides the following characterisation of when
the left degree is finite in terms of the kernel of the considered
freely irreducible morphism. This characterisation is proved in
\cite{CLT1} for irreducible morphisms with indecomposable domain or
with indecomposable codomain.

\begin{cor}
  \label{sec:spec-freely-irred-1}
  Let $X,Y\in \mathrm{mod}\,A$ be such that $X$ or $Y$ is
  indecomposable. Let $f\colon X\to Y$ be a freely irreducible
  morphism. Let $i\colon\mathrm{Ker}(f)\to X$ be the inclusion
  morphism. Then, for every integer $n\geqslant 1$, the following
  conditions are equivalent.
  \begin{enumerate}[(i)]
  \item $d_\ell(f)$ is finite and equal to $n$,
  \item $i\in \mathrm{rad}^n\backslash \mathrm{rad}^{n+1}$.
  \end{enumerate}
\end{cor}
\begin{proof}
  Both
  $(i)$ and $(ii)$ imply that $d_\ell(f)$ is finite. Hence $d_\ell(f)$
  may be assumed to be finite. Therefore, part (2) of
  Theorem~\ref{sec:finite-left-degree-1} applies with $f=f'$. This
  proves the equivalence.
\end{proof}

The following Corollary extends \cite[Corollary 3.2,
Corollary 3.8]{CLT1} to the setting of algebras over perfect
fields with an analogous proof. The details are given for the
convenience of the reader. The Corollary relates the finiteness of
left and right degrees for a given irreducible morphism.
\begin{cor}
  \label{sec:right-degree-versus}
  Let $f\colon X\to Y$ be an irreducible morphism with $X\in\mathrm{ind}\,A$ or $Y\in \mathrm{ind}\,A$.
  \begin{enumerate}
  \item If $d_\ell(f)<\infty$ then $f$ is not a monomorphism and
    $d_r(f)=\infty$. In particular every left minimal almost split
    morphism with non-injective domain has infinite left degree.
  \item If $d_r(f)<\infty$ then $f$ is not an epimorphism and
    $d_\ell(f)=\infty$.
    In particular every right minimal almost split morphism with
    non-projective codomain has infinite right degree.
  \item If $A$ is of finite representation type then the following
    conditions are equivalent
    \begin{enumerate}[(a)]
    \item $d_\ell(f)<\infty$,
    \item $d_r(f)=\infty$,
    \item $f$ is an epimorphism.
    \end{enumerate}
  \end{enumerate}
\end{cor}
\begin{proof}
  Assume that $d_\ell(f)<\infty$. By
  Theorem~\ref{sec:finite-left-degree-1}, there exists an irreducible
  morphism $f'\colon X\to Y$ such that the inclusion morphism $\mathrm{Ker}(f')\to X$ does not lie in $\mathrm{rad}^{1+d_\ell(f)}$. In
  particular, $f'$ is not a monomorphism, and hence $X$ has larger
  length than $Y$ does. Therefore $f$ is not a monomorphism, and hence
  it is an epimorphism. Using dual considerations shows that, if
  $d_r(f)<\infty$, then $f$ is a monomorphism and not an
  epimorphism. Whence (1) and (2). And (3) follows from (1) and (2)
  because, when $A$ is of finite representation type, then any nonzero
  morphism lies in $\mathrm{rad}^n\backslash\mathrm{rad}^{n+1}$ for some
  $n\in \mathbb N$.
\end{proof}

\subsection{Application to the kernel of irreducible morphisms}
\label{subsec_degreealmostsplit}

The objective of this subsection is to prove
Proposition~\ref{degree}: two irreducible morphisms $X\to Y$ where
$X\in \mathrm{ind}\,A$ is such that $\kappa_X\simeq \k$ have isomorphic
kernels (or, cokernels) whenever at least one of them has finite left
degree (or, right degree, respectively). For this purpose, the
following result is needed. It is proved in \cite[Proposition
3.3]{CLT1} when $\k$ is algebraically closed and part of the statement
is proved in
in \cite[1.10]{L2} when $\k$ is an artin ring.

\begin{prop}\label{degree}
  Let $f\colon X\to Y$ be an irreducible morphism with $Y\in\mathrm{ind}\,A$ non-projective.   Assume that the almost split sequence in
  $\mathrm{mod}\,A$

  \begin{equation}
    \xymatrix@R=10pt{
      && X' \ar[rd]^{f'}\\
      0\ar[r]&\tau_A Y \ar[ru]^g \ar[rd]_{g'} && Y\ar[r] & 0\\
      &&X \ar[ru]_f
    }\notag
  \end{equation}

  \noindent is such that $X'\neq 0$. Then $d_\ell(f)<\infty$ if and only if
  $d_\ell(g)<\infty$. In such a case $d_\ell(g)=d_\ell(f)-1$.
\end{prop}
\begin{proof}
  It was proved in \cite[1.10]{L2} that if $d_\ell(f)<\infty$ then
  $d_\ell(g)\leqslant d_\ell(f)-1$. Conversely assume that
  $d_\ell(g)=n<\infty$. Then there exists $Z\in \mathrm{ind}\,A$ and $h\in
  \mathrm{Hom}_A(Z,\tau_AY)$ such that $h\in \mathrm{rad}^n\backslash\mathrm{rad}^{n+1}$ and $hg\in \mathrm{rad}^{n+2}$. In particular $hg'\in
  \mathrm{rad}^{n+1}$. On the one hand $d_\ell([g,g'])=\infty$ because
  $[g,g']$ is left minimal  almost split
  (see Corollary~\ref{sec:right-degree-versus}). Therefore $h[g,g']=[hg,hg']\in \mathrm{rad}^{n+1}\backslash\mathrm{rad}^{n+2}$, and hence $hg'\not\in \mathrm{rad}^{n+2}$. On the other hand $hg'f=-hgf'\in \mathrm{rad}^{n+3}$. Thus $d_\ell(f)\leqslant n+1=d_\ell(g)+1$.
\end{proof}

Now it is possible to prove the main result of this subsection. It is
proved in \cite[Corollary 3.6]{CLT1} when $\k$ is algebraically
closed (under which condition the residue field of any indecomposable
module is trivial).

\begin{prop}
  Let $f_1,f_2\colon X\to Y$ be irreducible morphisms. Assume that
  $X\in \mathrm{ind}\,A$ and $\kappa_X\simeq \k$, or else $Y\in \mathrm{ind}\,A$ and $\kappa_Y\simeq \k$.
  If $d_\ell(f_1)<\infty$ then $d_\ell(f_1)=d_\ell(f_2)$ and $\mathrm{Ker}(f_1)\simeq \mathrm{Ker}(f_2)$.
\end{prop}
\begin{proof}
  First suppose that $X\in \mathrm{ind}\,A$ and $\kappa_X\simeq
  \k$. In particular, both $f_1$ and $f_2$ are freely
  irreducible. Denote by $i\colon \mathrm{Ker}(f_1)\to X$ and $j\colon \mathrm{Ker}(f_2)\to X$ the inclusion
morphisms. Let $Y=\oplus_t Y_t$ be a decomposition such that $Y_t\in
\mathrm{ind}\,A$ for every $t$, an write $f_1=[f_{1,t}\ ;\ t]$ and
$f_2=[f_{2,t}\ ;\ t]$ accordingly.
 For every $t$, the arrow $X\to Y_t$ in
$\Gamma(\mathrm{mod}\,A)$ has finite degree at most $d_\ell(f_1)$. Since
$\kappa_X\simeq \k$, it follows that $X\to Y_t$ has valuation $(b,1)$
for some $b\geqslant 1$ (see \ref{subsec_modules}). Therefore, there
exists $a_t\in \mathrm{Aut}_A(Y_t)$ such that $f_{2,t}-f_{1,t}a_t\in \mathrm{rad}^2$. Since $if_1=0$, it follows that $i[f_{1,t}a_t\ ;\ t]=0$,
and hence $if_2\in \mathrm{rad}^{2+d_\ell(f_1)}$. This shows that $f_2$
has finite left degree bound by $d_\ell(f_1)$. Now, $f_1$ and $f_2$
play symmetric roles. Similar considerations as above therefore show
that $jf_1\in \mathrm{rad}^{d_\ell(f_2)+2}$ and $d_\ell(f_1)\leqslant
d_\ell(f_2)$. Hence, $d_\ell(f_1)=d_\ell(f_2)$ and $jf_1\in \mathrm{rad}^{2+d_\ell(f_1)}$. Applying part (3) of
Theorem~\ref{sec:finite-left-degree-1}
to $f_1$ then shows that there exists a section $\sigma\colon \mathrm{Ker}(f_2)\to \mathrm{Ker}(f_1)$ such that $j - \sigma i\in \mathrm{rad}^{d_\ell(f_1)+1}$. In particular, $\sigma$ is an isomorphism
because both $\mathrm{Ker}(f_1)$ and $\mathrm{Ker}(f_2)$ are
indecomposable. This finishes proving the proposition when $X\in \mathrm{ind}\,A$ and $\kappa_X\simeq \k$.

\medskip

Now, assume that $Y\in \mathrm{ind}\,A$ and $\kappa_Y\simeq \k$. If
$d_\ell(f_1)=1$, then $f_1$
is right minimal almost split, and hence so is $f_2$. The conclusion
is then immediate. From now on, assume that $d_\ell(f_1)>1$. In
particular, neither $f_1$ nor $f_2$ is minimal right almost
split. From part (3) of Theorem~\ref{sec:finite-left-degree-1}, $f_1$
is not a monomorphism.  Hence, $Y$ is not projective, and there exist
almost
split sequences
  \begin{equation}
    \xymatrix@R=10pt{
      && X' \ar[rd]&&&&&& X' \ar[rd]\\
      0\ar[r]&\tau_AY \ar[ru]^{g_1} \ar[rd] && Y\ar[r] & 0&\mathrm{and}&0\ar[r]& \tau_AY
      \ar[ru]^{g_2} \ar[rd] && Y\ar[r] & 0 \\
      &&X \ar[ru]_{f_1}&&&&&&X \ar[ru]_{f_2}}
    \notag
  \end{equation}
where $X'\neq 0$.  First note that $\mathrm{Ker}(f_1)\simeq \mathrm{Ker}(g_1)$. Indeed, by diagram chasing, the natural monomorphism
  $\mathrm{Ker}(f_1)\to X$ induces an monomorphism $\mathrm{Ker}(f_1)\to
  \mathrm{Ker}(g_1)$. Moreover $\mathrm{Ker}(f_1)$ and $\mathrm{Ker}(g_1)$ have
  the same length. Therefore
  $\mathrm{Ker}(f_1)\simeq \mathrm{Ker}(g_1)$. Similar considerations show
  that $\mathrm{Ker}(f_2)\simeq \mathrm{Ker}(g_2)$. Next if follows from
  Proposition~\ref{degree}
  that $d_\ell(g_1)=d_\ell(f_1)-1$. Since $g_1,g_2\colon \tau_A Y\to X'$
  are irreducible, and since $\kappa_{\tau_AY}\simeq \kappa_Y\simeq
  \k$, the first part of the
  proof yields that $d_\ell(g_1)=d_\ell(g_2)$ and $\mathrm{Ker}(g_1)\simeq \mathrm{Ker}(g_2)$. As a consequence, $\mathrm{Ker}(f_1)\simeq \mathrm{Ker}(f_2)$. Finally, applying again
  Proposition~\ref{degree} gives $d_\ell(f_2)=d_\ell(g_2)+1$, and
  hence $d_\ell(f_2)=d_\ell(f_1)$.
\end{proof}

\subsection{Application to the finite representation type}
\label{subsec_finitetype}

To end this section, an application of
Theorem~\ref{sec:finite-left-degree-1} to a characterisation of
algebras of finite representation type is given below in terms of left
and right degrees.
In the  following theorem, the equivalence of conditions (a) to (e)
was proved in \cite[4]{CLT1}, and the assertions (f) and (g)
where proved in \cite{C4}, in the case where $\k$ is algebraically closed.

\begin{thm}
  Let $A$ be a finite dimensional algebra over a perfect field
  $\k$. The following conditions are equivalent
  \begin{enumerate}[(a)]
  \item $A$ is of finite representation type,
  \item for every indecomposable projective $A$-module $P$, the
    inclusion $\mathrm{rad}(P)\hookrightarrow P$ has finite right degree,
  \item for every indecomposable injective $A$-module $I$, the
    quotient $I\twoheadrightarrow I/\mathrm{soc}(I)$ has finite left
    degree,
  \item for every irreducible epimorphism $X\to Y$ with $X$ or $Y$
    indecomposable, the left degree is finite,
  \item for every irreducible monomorphism $X\to Y$ with $X$ or $Y$
    indecomposable, the right degree is finite.
  \end{enumerate}
  If $A$ is of finite representation type, then
  \begin{enumerate}[(a)]
    \setcounter{enumi}{5}
  \item there exists an indecomposable injective $A$-module $I$ with
    associated irreducible epimorphism $\pi\colon I\twoheadrightarrow
    I/\mathrm{soc}(I)$ such that $d_\ell(f)\leqslant d_\ell(\pi)$ for every irreducible epimorphism
    $f\colon X\to Y$ with $X$ or $Y$ indecomposable.
  \item there exists an indecomposable projective $A$-module $P$ with
    associated irreducible monomorphism $\iota\colon \mathrm{rad}(P)\hookrightarrow
    P$ such that $d_r(f)\leqslant d_r(\iota)$ for every irreducible monomorphism
    $f\colon X\to Y$ with $X$ or $Y$ indecomposable.
  \end{enumerate}
\end{thm}
\begin{proof}   The proof of the equivalence of conditions
  (a) to (e) follows from the considerations in
  \cite[Lemma 4.1, Lemma 4.2, Theorem A]{CLT1} provided that
 Theorem~\ref{sec:finite-left-degree-1} is used here instead of
  \cite[Proposition 3.4]{CLT1} there, and that
  Corollary~\ref{sec:right-degree-versus} is
  used here instead of \cite[Corollary 3.8]{CLT1} there.

  There only remains to prove (f) and (g) assuming that $A$ is of finite
  representation type. Since (g) is dual to (f) and since there are
  only finitely many isomorphism classes of indecomposable injective
  modules, it suffices to prove that, for every
  irreducible epimorphism $f\colon X\to Y$, there exists  an
  indecomposable injective module
  $I$ such that $d_\ell(f)\leqslant d_\ell(\pi)$. Apply
  Theorem~\ref{sec:finite-left-degree-1} to such an $f$: Since
  $d_\ell(f)<\infty$, there exists an
  irreducible morphism $f'\colon
  X\to Y$ such that $f-f'\in \mathrm{rad}^2$ and such that the inclusion
  morphism $i\colon \mathrm{Ker}(f')\to X$ lies in $\mathrm{rad}^{d_\ell(f)}\backslash\mathrm{rad}^{d_\ell(f)+1}$. Let $S$ be a
  simple direct summand of $\mathrm{soc}(\mathrm{Ker}(f'))$. Let $I$ be its injective hull. Therefore there
  exists a morphism $X\to I$ making the following diagram commute
\[
    \xymatrix{
      S\ar@{^(->}[rd] \ar@{^(->}[rrr]&&&I\\
      & \mathrm{Ker}(f') \ar@{^(->}_{i}[r] & X \ar[ru]&.
    }
\]
Note that all irreducible morphisms $I\to I/S$  have their kernel
isomorphic to $S$. Therefore,
considering (c), and applying
Theorem~\ref{sec:finite-left-degree-1} to $\pi\colon I\to I/S$ yields
that the inclusion morphism $S\to I$ lies in $\mathrm{rad}^{d_\ell(\pi)}\backslash\mathrm{rad}^{d_\ell(\pi)+1}$. Thus,
$d_\ell(f)\leqslant d_\ell(\pi)$.
\end{proof}

\section{Application to compositions of irreducible morphisms}
\label{sec:appl-comp-irred}

In this section, by a \emph{path} is meant a path of irreducible
morphisms between indecomposable modules. The objective of this
section is to investigate the paths $f_1,\ldots,f_n$ such that
$f_1\cdots f_n\in \mathrm{rad}^{n+1}$. The main result is
Theorem~\ref{sec:introduction}. The following proposition shows a
first part of it.

\begin{prop}
  \label{sec:appl-comp-irred-3}
  Let
  $X_0\xrightarrow{f_1}X_1\to \cdots \to X_{n-1}\xrightarrow{f_n}X_n$
  be a path. For each $t$, let $f_t'\colon X_{t- 1}\to X_t$ be such as
  $f'$ in Theorem B when $f=f_t$. The following are equivalent.
  \begin{enumerate}[(i)]
  \item $f_1\cdots f_n\in
    \mathrm{rad}^{n+1}$.
  \item There exists $t\in \{1,\ldots,n\}$ such that
    $d_\ell(f_t)\leqslant t-1$, and there exists $h\in \mathrm{rad}^{t-1-d_\ell(f_t)}(X_0,\mathrm{Ker}(f'_t))$ not lying in
    $\mathrm{rad}^{t-d_\ell(f_t)}$, and such that $f_1\cdots f_{t-1}-h i \in \mathrm{rad}^t$ (where $i\colon \mathrm{Ker}(f'_t)\to X_{t-1}$ is the
    inclusion morphism).
  \end{enumerate}
  In particular, if $f_1\cdots f_n\in \mathrm{rad}^{n+1}$ and if $t$ is an
  integer such as in $(ii)$, then the following holds.
  \begin{enumerate}[(i)]
    \setcounter{enumi}{2}
  \item There exists a path of length
    $t-1-d_\ell(f_t)$ from $X_0$ to $\mathrm{Ker}(f'_t)$ and with nonzero
    composition.
  \end{enumerate}
\end{prop}
\begin{proof}
  Assume $(ii)$. Since $if_t=0$, then
  \[
  f_1 \cdots f_n = (f_1 \cdots f_{t-1}-hi) \cdot f_t\cdots f_n\,.
  \]
  By assumption, $f_1 \cdots f_{t-1}-hi$ lies in $\mathrm{rad}^t$. Consequently, $f_1\cdots f_n\in \mathrm{rad}^n+1$. This
  proves that $(ii)$ implies $(i)$.

  \medskip

  Assume that $f_1\cdots f_n\in \mathrm{rad}^{n+1}$.  There is no loss of
  generality in assuming that $f_1\cdots f_{n-1}\not\in \mathrm{rad}^n$.
  Therefore $f_n$ has finite left degree. Now, consider the exact
  sequence obtained upon applying
  Theorem~\ref{sec:finite-left-degree-1} to $f:=f_n$.  This shows
  assertion $(ii)$. Thus, $(i)\implies (ii)$. The last assertion of
  the corollary is obtained by considering any decomposition of $h$
  into a sum of compositions of paths of length at least $d$.
\end{proof}

The following example shows that, in the previous result,  $(iii)$ need
not imply $(i)$.

\begin{ex} \label{hereditary} Consider the 
  artin $\mathbb{R}$-algebra
of finite representation type
  \[A= \left(
    \begin{tabular}{ll}
      $\mathbb{C}$& 0 \\
      $\mathbb{C}$ & $\mathbb{R}$
    \end{tabular}
  \right)
  \]
 where $\mathbb{R}$ is the field of real numbers and
  $\mathbb{C}$ the field of complex numbers.

  The Auslander-Reiten quiver without considering valuations in the arrows is the following

  \[
  \begin{tabular}{lllllllll}
    $S_2$ &  & $\dots$ &  & $P_{1}/S_{2}$  &  &  \\
          & $\stackrel{f} \searrow $   &  & $\stackrel{g} \nearrow $  &  & $\stackrel{h} \searrow $ &  \\
          &  & $P_1$ &  &  &  & $\tau ^{-1}P_{1}$ \\
          & $\stackrel{f^{\prime}} \nearrow $  &   & $\stackrel{g^{\prime}} \searrow$
                          &  & $\stackrel{h^{\prime}}{\nearrow}$ &  \\
    $S_2$ &  & $\dots$ &  & $P_{1}/S_{2}$ &  &
  \end{tabular}
  \]\vspace{.05in}

  \noindent where the two copies of $S_2$ and
  $P_{1}/S_{2}$ are identified.

  The path
  $S_2 \stackrel{f'} \longrightarrow P_1 \stackrel{g} \longrightarrow
  P_{1}/S_{2}$
  is a pre-sectional path, since $S_{2} {\oplus} S_{2}$ is a summand
  of the domain of the right almost split morphism for $P_1$.
  Moreover, $f^{\prime} g\in \mathrm{rad}^2 \backslash \mathrm{rad}^3$. In
  fact, if $f^{\prime} g \in \mathrm{rad}^3,$ since $fg =0$ then
  $d_l((f,f^{\prime}) )=1$ a contradiction to the fact that
  $(f, f^{\prime})$ is not a surjective right almost split morphism.

  Observe that the above  path satisfies $(iii)$ in Proposition~\ref{sec:appl-comp-irred-3}, but it is not in $\mathrm{rad}^3$.
\end{ex}

\begin{rem}
  \label{sec:appl-comp-irred-4}
  Let $X_0\xrightarrow{f_1} X_1\to \cdots \to X_{n-1}\xrightarrow{f_n}
  X_n$ be a path.
  \begin{enumerate}
  \item Assume that there exists an integer $t$ such that the composition of any
    path $X_0\to X_1\to \cdots \to X_{t-1}\to X_t$ is nonzero and such
    that $d_\ell(f_s)\geqslant
    s$ for every $s\geqslant t+1$. Then, the former
    condition entails
    that $f_1\cdots f_t\in \mathrm{rad}^t\backslash \mathrm{rad}^{t+1}$ (see
    \cite[Proposition 3]{CLT2}). And the latter condition  then implies
    that $f_1\cdots f_n\not\in \mathrm{rad}^{n+1}$.
  \item Following \cite[Section 3]{CLT2}, if there exists a path
    $X_0\to X_1\to \cdots \to X_{n-1}\to
    X_n$ with composition equal to $0$, and if each one of the
    $\k$-vector spaces $\mathrm{irr}(X_{i-1},X_i)$ is one dimensional,  then $f_1\cdots f_n\in \mathrm{rad}^{n+1}$.
  \end{enumerate}
\end{rem}

It is now straightforward to prove the second
main result of this text.
\begin{proof}[Proof of Theorem~\ref{sec:introduction}]
  The equivalence $(i)\Leftrightarrow (ii)$ is provided by Proposition~\ref{sec:appl-comp-irred-3}.

  Assume that $f_1\cdots f_n\in \mathrm{rad}^{n+1}$.
  Applying Proposition~\ref{sec:appl-comp-irred-3} yields an integer $t$
  such that $f_1\cdots f_t\in \mathrm{rad}^{t+1}$ and $d_\ell(f_t)\leqslant
  t-1$, and a path of length
  $t-1-d_\ell(f_t)$ from $X_0$ to $\mathrm{Ker}(f_t)$ and with nonzero
  composition. Applying
  \cite[Section 3]{CLT2} to the path $(f_1,\ldots,f_t)$ then yields a path
  $X_0\to \cdots \to X_t$ with zero composition. This shows that
  $(i)$ and $(ii)$ imply $(iii)$.

  Assume $(iii)$ and assume that
  $\mathrm{dim}_\k\mathrm{irr}(M_{s-1},M_s)=1$ for every
  $s\in \{1,\ldots,t\}$. Then $f_1\cdots f_t\in \mathrm{rad}^{t+1}$
  according to part (2) of Lemma~\ref{sec:appl-comp-irred-1}.
\end{proof}

\bibliographystyle{plain}
\bibliography{biblio-CLT16}

\begin{thebibliography}{10}

\bibitem{ARS}
M.~Auslander, I.~Reiten, and S.~O. Smal{\o}.
\newblock {\em Representation theory of {A}rtin algebras}, volume~36 of {\em
  Cambridge Studies in Advanced Mathematics}.
\newblock Cambridge University Press, Cambridge, 1995.

\bibitem{Bre}
S.~Brenner.
\newblock On the kernel of an irreducible map.
\newblock {\em Linear Algebra Appl.}, 365:91--97, 2003.
\newblock Special issue on linear algebra methods in representation theory.

\bibitem{C4}
C.~Chaio.
\newblock Problems solved by using degrees of irreducible morphisms.
\newblock In {\em Expository lectures on representation theory}, volume 607 of
  {\em Contemp. Math.}, pages 179--203. Amer. Math. Soc., Providence, RI, 2014.

\bibitem{CLT1}
C.~Chaio, P.~Le~Meur, and S.~Trepode.
\newblock Degrees of irreducible morphisms and finite representation type.
\newblock {\em J. Lond. Math. Soc. (2)}, 84(1):35--57, 2011.

\bibitem{CLT2}
C.~Chaio, P.~Le~Meur, and S.~Trepode.
\newblock Covering techniques for {A}uslander-{R}eiten theory.
\newblock {\em J. Pure Appl. Algebra}, 2018.
\newblock https://doi.org/10.1016/j.jpaa.2018.04.013.

\bibitem{CD}
F.~U. Coelho and D.~D. da~Silva.
\newblock Relative degrees of irreducible morphisms.
\newblock {\em J. Algebra}, 428:471--489, 2015.

\bibitem{IT}
K.~Igusa and G.~Todorov.
\newblock A characterization of finite {A}uslander-{R}eiten quivers.
\newblock {\em J. Algebra}, 89(1):148--177, 1984.

\bibitem{MR748231}
K.~Igusa and G.~Todorov.
\newblock Radical layers of representable functors.
\newblock {\em J. Algebra}, 89(1):105--147, 1984.

\bibitem{Kra}
H.~Krause.
\newblock The kernel of an irreducible map.
\newblock {\em Proc. Amer. Math. Soc.}, 121(1):57--66, 1994.

\bibitem{L}
S.~Liu.
\newblock Degrees of irreducible maps and the shapes of {A}uslander-{R}eiten
  quivers.
\newblock {\em J. London Math. Soc. (2)}, 45(1):32--54, 1992.

\bibitem{L2}
S.~Liu.
\newblock Shapes of connected components of the {A}uslander-{R}eiten quivers of
  {A}rtin algebras.
\newblock In {\em Representation theory of algebras and related topics
  ({M}exico {C}ity, 1994)}, volume~19 of {\em CMS Conf. Proc.}, pages 109--137.
  Amer. Math. Soc., Providence, RI, 1996.

\end{thebibliography}

\end{document}